\documentclass[a4paper,10pt]{amsart}

\usepackage[colorlinks=printing]{hyperref}
\usepackage{amssymb,amsmath,esint,graphicx,subfigure}

\newtheorem{theorem}{Theorem}[section]
\newtheorem{lemma}[theorem]{Lemma}
\newtheorem{proposition}[theorem]{Proposition}
\newtheorem{corollary}[theorem]{Corollary}
\newtheorem{definition}{Definition}[section]

\newcommand{\del}{\partial}

\newcommand{\ra}{\rightarrow}

\newcommand{\alp}{\alpha}
\newcommand{\eps}{\ensuremath{\varepsilon}}
\newcommand{\R}{\ensuremath{\mathbb{R}}}

\newcommand{\Levy}{\ensuremath{\mathcal{L}}}

\newcommand{\dif}{\mathrm{d}}

\theoremstyle{remark} \newtheorem{remark}[theorem]{Remark} \theoremstyle{definition} 

\numberwithin{equation}{section}

\begin{document}

\title[Periodic fractional conservation laws]{On the spectral
  vanishing viscosity method for periodic fractional conservation laws}

\author[{S.~Cifani}]{{Simone Cifani}}
\address[Simone Cifani]{\\ Department of Mathematics\\ Norwegian University of Science and Technology (NTNU)\\
 N-7491 Trondheim, Norway}
 \email[]{simone.cifani\@@math.ntnu.no}
\urladdr{http://www.math.ntnu.no/\~{}cifani/}

\author[{E.R.~Jakobsen}]{{Espen R. Jakobsen}}
\address[Espen R. Jakobsen]{\\ Department of Mathematics\\ Norwegian University of Science and Technology (NTNU)\\
 N-7491 Trondheim, Norway}
\email[]{erj\@@math.ntnu.no}
\urladdr{http://www.math.ntnu.no/\~{}erj/}

\keywords{Fractional/fractal conservation laws, entropy solutions,
  Fourier spectral methods, spectral vanishing viscosity, convergence,
  error estimate.}
\subjclass[2010]{
65M70, %    Spectral, collocation and related methods
35K59, %    Quasilinear parabolic equations
35R09; %    Integro-partial differential equation
65M15, %    Error bounds
65M12, %    Stability and convergence of numerical methods
35K57, %    Reaction-diffusion equations
35R11. %    Fractional partial differential equations
}

\thanks{This research was supported by the Research Council of Norway (NFR) through the project "Integro-PDEs:
Numerical Methods, Analysis, and Applications to Finance". }

\begin{abstract}
We introduce and analyze a spectral vanishing viscosity approximation of periodic fractional conservation laws. The fractional part of these equations can be a
fractional Laplacian or other non-local operators that are generators of pure jump L\'{e}vy processes. To accommodate for shock solutions, we first extend to
the periodic setting the Kru\v{z}kov-Alibaud entropy formulation and prove well-posedness. Then we introduce the numerical method, which is a non-linear
Fourier Galerkin method with an additional spectral viscosity term. This type of approximation was first introduced by Tadmor for pure conservation laws. We
prove that this {\em non-monotone} method converges to the entropy solution of the problem, that it retains the spectral accuracy of the Fourier method, and
that it diagonalizes the fractional term reducing dramatically the computational cost induced by this term. We also derive a robust $L^1$-error estimate, and
provide numerical experiments for the fractional Burgers' equation.
\end{abstract}

\maketitle

\section{Introduction}
In this paper we are concerned with a spectral vanishing viscosity
(henceforth SVV) approximation for periodic solutions of non-local or
fractional conservation laws of the form
\begin{equation*}
\left\{
\begin{array}{ll}
\partial_tu+\partial_x\cdot f(u)=-(-\Delta)^{\lambda/2}u,&(x,t)\in D_T\\
u(x,0)=u_{0}(x), &x\in\Lambda,
\end{array}
\right.
\end{equation*}
for $\lambda\in(0,2)$, or more generally, for
\begin{equation}\label{1}
\left\{
\begin{array}{ll}
\partial_tu+\partial_x\cdot f(u)=\Levy^\mu[u],&(x,t)\in D_T\\
u(x,0)=u_{0}(x), &x\in\Lambda,
\end{array}
\right.
\end{equation}
where $D_T=\Lambda\times(0,T)$ and $\Lambda=(0,2\pi)^d$, and $\Levy^\mu[\cdot]$ is a non-local (L\'{e}vy type) operator defined as
\begin{align}\label{non-loc-op}
\Levy^\mu[\phi(\cdot)](x)=\int_{|z|>0}\phi(x+z)-\phi(x)-z\cdot\partial_x\phi(x)\,\mathbf{1}_{|z|<1}\ \dif\mu(z),
\end{align}
where $\mathbf{1}(\cdot)$ is the indicator function. Throughout the paper we assume that
\begin{align}
&f=(f_1,\ldots,f_d)\text{ with }f_j\in C^{s}(\R)\text{ for all } j=1,\ldots,d \text{ ($s$ to be defined)};\label{flux}\tag{\textbf{A}.1}\\
&\text{$\mu\geq 0$ is a Radon measure such that $\int_{|z|>0}|z|^2\wedge 1 \ \dif\mu(z)<\infty$};\label{levy-measure}\tag{\textbf{A}.2}\\
&u_0\in L^\infty(\Lambda)\cap BV(\Lambda), \text{ $u_0$ is $\Lambda$-periodic}.\label{in_condition}\tag{\textbf{A}.3}
\end{align}
Here and in the rest of the paper, $a\wedge b=\min(a,b)$,
\begin{equation*}
\begin{split}
\partial_t=\frac{\partial}{\partial
  t},\quad\partial_j=\frac{\partial}{\partial
  x_j}\quad\text{and}\quad\partial_x=(\partial_1,\partial_2,\ldots,\partial_d).
\end{split}
\end{equation*}

Integro-PDEs like \eqref{1} typically model anomalous
convection-diffusion phenomena. When $\mu=\pi_\lambda$ is defined by
\begin{equation}\label{frac_lap}
\begin{split}
d\pi_{\lambda}(z)=\frac{c_\lambda}{|z|^{d+\lambda}}dz,\qquad c_\lambda>0\text{ and }\lambda\in(0,2),
\end{split}
\end{equation}
then $\Levy=-(-\Delta)^{\lambda/2}$ and the equation finds applications in e.g.~over-driven detonation in gases \cite{Clavin} and anomalous diffusion in
semiconductor growth \cite{Woyczynski}. Applications in dislocation dynamics, hydrodynamics and molecular biology can also be found, see e.g.~the references in
\cite{Alibaud,Droniou/Gallouet/Vovelle}. Many more applications can be found if asymmetric measures $\mu$ are allowed. For example, we cover the (linear)
option pricing equations for all L\'{e}vy models used in mathematical finance \cite{CT:Book,Schoutens}, if
\begin{equation}\label{asym}
\begin{split}
d\mu(z)=g(z)\,d\pi_{\lambda}(z),
\end{split}
\end{equation}
for some possibly asymmetric locally Lipschitz continuous
function $g(\cdot)$. An
example is the one-dimensional ($d=1$) CGMY model where
\begin{equation*}
g(z)=\left\{
\begin{split}
&Ce^{-G|z|}&\text{for $z>0$,}\\
&Ce^{-M|z|}&\text{for $z<0$,}
\end{split}
\right.
\end{equation*}
for positive constants $C,G,M$ (and $Y=\lambda$). In general the non-local operator $\Levy$ is the generator of a pure jump L\'{e}vy process, and conversely,
any L\'{e}vy process will have generator like $\Levy$ when \eqref{levy-measure} is satisfied. We refer to the books \cite{Ap:Book,CT:Book,Sato} for more
information about L\'{e}vy processes and their many applications. The most general L\'{e}vy measures for which the results of this paper applies, are L\'{e}vy
measures $\mu$ that can be decomposed as
\begin{equation}\label{splitt}
\begin{split}
\mu=\mu_s+\mu_n,
\end{split}
\end{equation}
where
 \begin{equation}\label{rev_2}
\begin{split}
\mu_s,\mu_n\geq0,\quad \mu_s \text{ is symmetric$\quad$and}\quad\int_{|z|>0}|z|\wedge1\ \dif \mu_n(z)<\infty.
\end{split}
\end{equation}
See Section \ref{sec:asymm_approx} for statements of results and remarks. This class possibly includes all L\'{e}vy measures, but we have so far not found a
proof of this. At least it includes all the L\'{e}vy measures found in finance, see Remark \ref{rem8}, and also many singular measures like
e.g.~delta-measures.

It is important to note that non-linear equations like \eqref{1} do not admit classical solutions in general, and that shock discontinuities can develop even
from regular initial conditions. This is well known for pure conservations laws (where $\Levy=0$), see e.g.~\cite{Holden/Risebro}. For fractional conservation
laws where $\Levy=-(-\Delta)^{\lambda/2}$, it is shown in recent works that solutions are smooth for $\lambda\in[1,2)$
\cite{ChCz09,Droniou/Gallouet/Vovelle,KNS08}. However, when $\lambda\in(0,1)$, the fractional diffusion is too weak to prevent shock discontinuities from
forming, see \cite{Alibaud,DDL09,KNS08}. In some cases however, these shocks are smoothed out over time \cite{Chan}. When shocks form, weak solutions become
non-unique and entropy conditions are needed to select the physically correct solution -- the entropy solution. The well known Kru\v{z}kov entropy solution
theory for conservation laws was extended to fractional conservation laws in \cite{Alibaud}. This extension relies on new ideas for the fractional term and is
strongly influenced by the viscosity solution theory for fractional Hamilton-Jacobi-Bellman equations. Extensions of the Kru\v{z}kov-Alibaud theory to general
L\'{e}vy operators and even non-linear fractional terms can be found in \cite{CiJa10,KU09}.

In this paper we deal with a SVV (spectral vanishing viscosity) approximation of $\Lambda$-periodic entropy solutions of \eqref{1}. The method is a Fourier
Galerkin method with an additional spectral viscosity term. Because of the formation of shocks in the solutions of \eqref{1}, it is very difficult to devise a
convergent {\em and} spectrally accurate numerical approximation of this equation. This has to do with the fact that Fourier spectral methods support spurious
Gibbs oscillations, and thus fails to converge strongly toward discontinuous solutions. It is well known that such methods need to be augmented by some kind of
vanishing viscosity in order to achieve convergence. But the standard vanishing viscosity method is  not spectrally accurate. To overcome these problems, we
use the SVV approximation developed by Tadmor in \cite{Tadmor}, cf.~also \cite{Tadmor_MAIN,Maday/Tadmor,Schochet,Tadmor_2} and the books
\cite{Benyu,Canuto/Quarteroni}. To suppress spurious oscillations without sacrificing the overall spectral accuracy of the method, Tadmor adds a modified
viscosity term, which in Fourier space only affects high frequencies. There are two parameters involved in this approach, the coefficient of the viscosity term
$\eps$ and the size $m$ of the viscosity free spectrum. Spectral accuracy and convergence toward the unique, possibly discontinuous, entropy solution, then
follows by imposing appropriate conditions on $\epsilon$ and $m$. We also like to mention another important feature of the method. In all cases, it
diagonalizes the fractional term and hence reduces dramatically the computational cost induced by this term. In our rather naive implementation for the
fractional Burgers' equation, the SVV method turned out to be orders of magnitude faster than a Discontinuous Galerkin approximation of the same equation where
the fractional term gives full matrices.

When equation \eqref{1} is linear,  $f(u)=u$, or when it is local $\Levy=0$, there is a vast literature on numerical methods and analysis, some methods and
many references can be found e.g.~in \cite{Benyu,CT:Book,Holden/Risebro,Schoutens}. In the general case however, there is not much work on numerical methods,
we only know of the papers \cite{CiJaKa10,Dr10,DeRo04}.  Difference methods are introduced in \cite{Dr10} for equation \eqref{1}, and in \cite{DeRo04} for an
equation similar to \eqref{1} from radiation hydrodynamics. In \cite{Dr10}, the first general convergence result for monotone schemes is obtained. Finally, in
\cite{CiJaKa10}, a Discontinuous Galerkin approximation of \eqref{1} is analyzed and a Kuznetsov type of theory is established and used to derive error
estimates. A periodic extension of this theory will be used to find error estimates in this paper.

Throughout the paper we will use the following additional notation. A subscript $p$ indicates $\Lambda$-periodicity in the space variables (i.e.~in
$L^\infty_p$ or  $C^\infty_p$). Here $\Lambda$-periodic means $2\pi$-periodic in each coordinate direction.  As a generic constant we use $C$. Note that the
value of $C$ may change from line to line and expression to expression. We also need notation for high order derivatives and their norms. Let
$\alp=(\alp_1,\dots,\alp_d)$ be a multi index, then
$$ \partial_x^\alp=\partial_1^{\alp_1}\partial_2^{\alp_1}\cdots\partial_d^{\alp_d},
\quad\partial_x^s=\bigcup_{|\alp|=s}\Big\{\del_x^\alp\Big\},\quad\text{and}\quad\|\del^s_x\phi\|_{L^p}^p=\sum_{|\alp|=s}\|\del_x^\alp\phi\|_{L^p}^p.$$
Remember that $\alp_j\geq0$, $|\alp|=\alp_1+\dots+\alp_d$, and that
$x^\alp=x_1^{\alp_1}\cdots x_d^{\alp_d}$ for any $x\in\R^d$.

The rest of this paper is organized as follows. In Section \ref{sec:ent} we introduce an entropy formulation for periodic solutions \eqref{1}, and give a
$L^1$-contraction and uniqueness result. In the same section we introduce the classical vanishing viscosity approximation of \eqref{1}, and show convergence
towards \eqref{1} with optimal $L^1$ error estimate. As a corollary we get existence for \eqref{1}. The proofs rely on the Kru\v{z}kov's doubling of variables
device \cite{Alibaud,Kruzkov} and Kuznetsov type of arguments \cite{CiJaKa10,Kuznetsov}, and are given in the Appendix. The SVV approximation of \eqref{1} is
introduced in Section \ref{sec:approx}, and we show that it is
spectrally accurate and that it diagonalizes the non-local
operator. In sections \ref{sec:trunc_err}--\ref{sec:conv_anal}, we assume that
the measure $\mu$ is symmetric. In Section \ref{sec:trunc_err} we prove an energy estimate for the SVV method. Along with results from \cite{Tadmor_MAIN}, this
allows us to control the so-called "truncation error", the spectral projection error coming from the non-linear term.  In Section \ref{sec:aprior} we prove a
priori $L^\infty$, $BV$, and time regularity estimates for the SVV method, and obtain compactness. In Section \ref{sec:conv_anal} we prove that the SVV method
converge to the classical vanishing viscosity method from Section \ref{sec:ent}. Combined with the results of that section, it follows that the SVV method
converges to the entropy solution of \eqref{1}. In the process, we also prove the optimal $L^1$-rate of convergence for our SVV approximation. We solve
numerically, using our SVV method, the the fractional Burgers' equation in Section \ref{sec:num_ex}. Finally, in Section \ref{sec:asymm_approx} we extend the
results in the previous sections to allow for asymmetric measures $\mu$.

\section{Entropy formulation for periodic solutions}\label{sec:ent}
In this section we introduce an entropy formulation for $\Lambda$-periodic
solutions of the initial value problem \eqref{1}. To this end, we
write the operator $\Levy^\mu[\cdot]$ as
\begin{equation*}
\begin{split}
\Levy^\mu[\phi]=\Levy_r^\mu[\phi]+\Levy^{\mu,r}[\phi]-\gamma_\mu^r\cdot\partial_x\phi,
\end{split}
\end{equation*}
where
\begin{equation*}
\begin{split}
\Levy_{r}^\mu[\phi(\cdot)](x)&=\int_{|z|\leq r}\phi(x+z)-\phi(x)-z\cdot\partial_x\phi(x)\,\mathbf{1}_{|z|<1}\ \dif\mu(z),\\
\Levy^{\mu,r}[\phi(\cdot)](x)&=\int_{|z|>r}\phi(x+z)-\phi(x)\ \dif\mu(z),\\
\gamma_\mu^r&=\int_{r<|z|<1}z\ \dif\mu(z).
\end{split}
\end{equation*}
If $r>1$, we take $\gamma_\mu^r=0$. The adjoint of $\Levy^\mu[\cdot]$ takes the form
\begin{equation*}
\begin{split}
\Levy^{\ast,\mu}_r[\phi]=\Levy_r^{\ast,\mu}[\phi]+\Levy^{\ast,\mu,r}[\phi]+\gamma_\mu^r\cdot\partial_x\phi,
\end{split}
\end{equation*}
where
\begin{equation*}
\begin{split}
\Levy_{r}^{\ast,\mu}[\phi(\cdot)](x)&=\int_{|z|\leq r}\phi(x-z)-\phi(x)+z\cdot\partial_x\phi(x)\,\mathbf{1}_{|z|<1}\ \dif\mu(z),\\
\Levy^{\ast,\mu,r}[\phi(\cdot)](x)&=\int_{|z|>r}\phi(x-z)-\phi(x)\ \dif\mu(z).
\end{split}
\end{equation*}
We also let $\eta$, $\eta'$, and $q$ denote the functions
\begin{equation*}
\begin{split}
\eta(u,k)=|u-k|,\quad\eta'(u,k)=\mathrm{sgn}(u-k),\quad q_{j}(u,k)=\eta'(u,k)\,(f_j(u)-f_j(k)).
\end{split}
\end{equation*}
We now define the solution concept we will use in this paper.

\begin{definition}\emph{(Periodic entropy solutions)}\label{def:alibaud_periodic}
A function $u$ is a periodic entropy solution of the initial value problem \eqref{1} provided that
\begin{itemize}
\item[\emph{i)}]$u\in C([0,T];L_p^{\infty}(\R^d))$;
\item[\emph{ii)}]for all $k\in\R$, all $r>0$, and all nonnegative test functions $\varphi\in
C_p^{\infty}(\R^d\times(0,T))$,
\begin{equation}\label{entropy_ineq}
\begin{split}
&\iint_{D_T}\eta(u,k)\,\partial_t\varphi+q(u,k)\cdot\partial_x\varphi\\
&\qquad+\eta(u,k)\,\Levy^{\ast,\mu}_r[\varphi]+\eta'(u,k)\,\Levy^{\mu,r}[u]\,\varphi+\eta(u,k)\,\gamma_\mu^r\cdot\partial_x\varphi\ \dif x\,\dif t\geq0;
\end{split}
\end{equation}
\item[\emph{iii)}]$\mathrm{esslim}_{t\rightarrow0}\|u(\cdot,t)-u_0(\cdot)\|_{L^1(\Lambda)}=0$.
\end{itemize}
\end{definition}
\begin{remark}
In the entropy inequality \eqref{entropy_ineq} it is easy to see that
all the terms, except possibly the $\Levy^{\mu,r}$-term, are well defined and
$\Lambda$-periodic in view of \emph{i)}. The problem with the
$\Levy^{\mu,r}$-term is that we integrate a Lebesgue measurable
function w.r.t. a Radon measure $\mu$. But the term is still well defined
because the integrand of $\Levy^{\mu,r}[u]$ is measurable w.r.t.~the product
measure $\dif \mu(z)\dif x\dif t$. This is true because the integrand is
the $\dif \mu(z)\dif x\dif t$-a.e.~limit of continuous functions, a
fact which  readily follows from the
fact that $u$ is the $\dif x\dif t$-a.e.~limit of smooth functions.
By \emph{i)}, ({\bf A}.3), and Fubini, we then find that $\Levy^{\mu,r}[u]\in
C([0,T];L^\infty_p(\R^d))$.
\end{remark}

 We now state the
following central result:
\begin{theorem}\emph{($L^1$-contraction)}\label{th:uniqueness}
Let $u$ and $v$ be two entropy solutions of the initial value problem \eqref{1} with initial data $u_0$ and $v_0$. Then, for a.e.~$t\in(0,T)$,
\begin{align}\label{contraction_periodic}
\|u(\cdot,t)-v(\cdot,t)\|_{L^1(\Lambda)}\leq\|u_{0}-v_{0}\|_{L^1(\Lambda)}.
\end{align}
\end{theorem}
The proof will be given in Appendix \ref{app:1}.
Uniqueness for periodic entropy solutions of \eqref{1} immediately
follows by setting $u_0=v_0$.
\begin{corollary}\emph{(Uniqueness)} There is at most one entropy solution of \eqref{1}.
\end{corollary}

We now consider the vanishing viscosity approximation of \eqref{1},
\begin{equation}\label{prob:viscous}
\left\{
\begin{array}{ll}
\partial_tu_\epsilon+\partial_x\cdot f(u_\epsilon)=\Levy^\mu[u_\epsilon]+\epsilon\,\Delta u_\epsilon&(x,t)\in D_T,\\
u_\epsilon(x,0)=u_{0}(x)&x\in\Lambda.
\end{array}
\right.
\end{equation}
In this paper we always assume that this problem admits a unique classical solution $u_\epsilon$. This is of course true, but a proof lays outside the scope of
this paper. Remark \ref{remVV} below provides some ideas on how to prove this result. We now give an estimate on the rate of convergence of $u_\eps$ toward the
entropy solution $u$ of \eqref{1}.
\begin{theorem}[Convergence rate I]\label{lemma_rate}
Let $u$ be the periodic entropy solution of \eqref{1}, and $u_{\epsilon}$ be a
smooth solution of \eqref{prob:viscous}. Then,
\begin{equation}\label{j4}
\begin{split}
\|u(\cdot,t)-u_{\epsilon}(\cdot,t)\|_{L^1(\Lambda)}\leq C\,\sqrt{\epsilon}.
\end{split}
\end{equation}
\end{theorem}
The proof is given in Appendix \ref{app:2}. This result generalizes to periodic fractional conservation laws  Kuznetsov's well known result for scalar
conservation laws \cite{Kuznetsov}. As a by-product of the well-posedness of \eqref{prob:viscous} and Theorem \ref{lemma_rate}, we have the existence of
entropy solutions of \eqref{1}.

\begin{corollary}\emph{(Existence)} There exists an entropy solution of \eqref{1}.
\end{corollary}
\begin{remark}
\label{remVV} Uniqueness of solutions of \eqref{prob:viscous} can be proved using an entropy formulation (see the start of Appendix \ref{App:VV}) and a
standard adaptation of the proof of Theorem \ref{th:uniqueness} incorporating ideas of Carrillo \cite{Ca99} to handle the Laplace term. Existence of an entropy
solution can be proven e.g.~by appropriately modifying our spectral approximation, compactness, and convergence analysis, see the following sections. The
solution of \eqref{prob:viscous} will also be smooth. To see this, note that  the principal
  term in \eqref{prob:viscous} is the $\epsilon\Delta$-term while the
  $\Levy^\mu$-term is of lower order, and hence regularity proofs  for
  viscous conservation laws (\eqref{prob:viscous} with $\mu\equiv0$
  and $\eps>0$)  should still work after some modifications. We refer
  to e.g.~\cite{Malek} for regularity of viscous conservation laws,
  and note that the modifications typically consist of using
  interpolation inequalities for the $\Levy^\mu$-term, see e.g.~Lemma
  2.2.1 in \cite{Garroni}.
\end{remark}

\section{The spectral vanishing viscosity method}\label{sec:approx}
We introduce a Fourier spectral method for the $d$-periodic initial
value problem \eqref{1}. The approximate solutions will be $N$-trigonometric
polynomials,
\begin{equation*}
\begin{split}
u_{N}(x,t)=\sum_{|\xi|\leq N}\hat{u}_{\xi}(t)\,e^{i\xi\cdot x},
\end{split}
\end{equation*}
which solve the semi-discrete spectral vanishing viscosity (SVV) approximation
\begin{equation}\label{2}
\begin{split}
\partial_t u_N+\partial_x\cdot P_Nf(u_N)=\Levy^\mu[u_N]+\epsilon_N\sum_{j,k=1}^d\partial_{jk}^2Q_N^{j,k}\ast u_N
\end{split}
\end{equation}
with
\begin{equation}\label{j5}
\begin{split}
u_N(x,0)=P_Nu_0(x),
\end{split}
\end{equation}
where the Fourier projection $P_N$ is defined as
$$P_N \phi(x)=\sum_{|\xi|\leq N}\hat{\phi}_{\xi}\,e^{i\xi\cdot x}\qquad\text{for}\qquad
\hat{\phi}_{\xi}=\frac{1}{(2\pi)^d}\int_{\Lambda} \phi(x)\,e^{-i\xi\cdot x}\ \dif x.$$
The (spectral) vanishing viscosity term has the following three ingredients:
\medskip
\begin{itemize}
\item[(\textbf{A}.4)] a vanishing viscosity amplitude $\epsilon_N\sim N^{-\theta}$ with $0<\theta<1$;
\medskip
\item[(\textbf{A}.5)] a viscosity-free spectrum $m_N\sim N^{\frac{\theta}{2}}(\log N)^{-\frac{d}{2}}$;
\medskip
\item[(\textbf{A}.6)] a family of viscosity kernels
\begin{align*}
Q_N^{j,k}(x,t)=\sum_{p= m_N}^N\hat{Q}_p^{j,k}(t)\sum_{|\xi|=p}e^{i\xi\cdot x}
\end{align*}
satisfying
\smallskip
\begin{itemize}
\item $\hat{Q}_p^{j,k}$ is monotonically $p$-increasing,
\smallskip
\item $\hat{Q}_p^{j,k}$ spherically symmetric,
  $\hat{Q}_\xi^{j,k}=\hat{Q}_p^{j,k}$ for all $|\xi|=p$,
\smallskip
\item $|\hat{Q}_p^{j,k}-\delta_{jk}|\leq C\,m_N^2\,p^{-2}$ for all
  $p\geq m_N$.
\end{itemize}
\end{itemize}
\medskip
\noindent Such kernels can be conveniently implemented in Fourier space,
\begin{equation*}
\begin{split}
\sum_{j,k=1}^d\partial_{jk}^2Q_N^{j,k}\ast u_N=-\sum_{|\xi|=m_N}^N\left(\sum_{j,k=1}^d\hat{Q}_\xi^{j,k}(t)\,\xi_j\,\xi_k\right)\hat u_\xi(t)\,e^{i\xi\cdot x}.
\end{split}
\end{equation*}
Combined with one's favorite ODE solver (e.g.~Euler, Runge-Kutta, etc.), \eqref{2} and \eqref{j5} give a fully discrete numerical
 approximation method for \eqref{1}.

With left-hand sides set to zero ($\mu\equiv0$ and $\eps_N=0$), \eqref{2} becomes the standard Fourier approximation of \eqref{1}. It is well known that this
approximation is spectrally accurate but, as opposed to the equation, it lacks entropy dissipation. The approximation supports spurious Gibbs oscillations
which prevent strong convergence toward solutions containing shock discontinuities. If the $\Levy^\mu$-term is present in the equations, shock solutions are
still possible in some situations \cite{Alibaud/Droniou/Vovelle}, and the problem of the Gibbs oscillations remains.  In order to suppress such oscillations
without sacrificing the overall spectral accuracy of the method, we have followed Tadmor \cite{Tadmor} and added a vanishing spectral viscosity term to the
scheme, $\epsilon_N\sum_{j,k=1}^d\partial_{jk}^2Q_N^{j,k}\ast u_N$.

An important feature of Fourier method \eqref{2} is that it {\em diagonalizes}, and hence {\em localizes}, the non-local operator $\Levy^\mu[\cdot]$! This
leads to dramatically reduced computational cost for this term. Indeed,
\begin{equation}\label{u1}
\begin{split}
\Levy^\mu[u_N]=\sum_{|\xi|\leq N}G^\mu(\xi)\,\hat{u}_\xi(t)\,e^{i\xi\cdot x},
\end{split}
\end{equation}
where
\begin{equation}\label{weights}
\begin{split}
G^\mu(\xi)=\int_{|z|>0}e^{i\xi\cdot z}-1-i\xi\cdot z\,\mathbf{1}_{|z|<1}\ \dif\mu(z).
\end{split}
\end{equation}
Furthermore, when the measure $\mu$ is {\em symmetric},
\begin{align}\label{symm}
\mu(-B)=\mu(B)\quad\text{for all
Borel sets} \quad B\in\R^d\setminus\{0\},
\end{align}
the weights \eqref{weights} are all {\em real and non-positive}. This
follows since the imaginary part of the integrand is odd and the real
part is even and non-positive ($e^{i\xi\cdot z}=\cos(\xi\cdot
z)+i\,\sin(\xi\cdot z)$). Finally, we stress that the approximation of the
non-local operator \eqref{non-loc-op} is {\em spectrally
accurate} since, by Taylor's formula,
\begin{equation*}
\begin{split}
&\|\Levy^\mu[u_N(\cdot,t)]-\Levy^\mu[u(\cdot,t)]\|_{L^2(\Lambda)}\\
&\qquad\qquad\leq C\,\bigg(\sup_{j,k}\|\del_j\del_k(u_N-u)(\cdot,t)\|_{L^2(\Lambda)}+\|(u_N-u)(\cdot,t)\|_{L^2(\Lambda)}\bigg).
\end{split}
\end{equation*}

Now we define
\begin{equation*}
\hat R_\xi^{j,k}(t)=
\begin{cases}
\delta_{jk} &|\xi|\leq m_N,\\
\delta_{jk}-\hat Q_\xi^{j,k}(t) &|\xi|>m_N,
\end{cases}
\qquad
R_N^{j,k}(x,t)=\sum_{|\xi|\leq N}\hat R_\xi^{j,k}(t)\,e^{i\xi\cdot x},
\end{equation*}
and note that
\begin{equation}\label{full_lap}
\begin{split}
\Delta u_N(\cdot,t)=\sum_{j,k=1}^d\del_{j}\del_kQ_N^{j,k}(\cdot,t)\ast u_N(\cdot,t)+\sum_{j,k=1}^d\del_j\del_kR_N^{j,k}(\cdot,t)\ast u_N(\cdot,t).
\end{split}
\end{equation}
To conclude this section, we recall that by Lemma 3.1 and Corollary 3.2 of \cite{Tadmor_MAIN}, the spectral vanishing viscosity term is  an $L^p$-bounded
perturbation of the standard vanishing viscosity $\epsilon_N\Delta
u_N$:
\begin{lemma}\label{lem:tad}
For $0\leq r\leq s\leq 2$,
\begin{equation}\label{hh1}
\begin{split}
\left\|\sum_{j,k=1}^d\partial_{j}^r\partial_k^{s-r} R_N^{j,k}(\cdot,t)\right\|_{L^1(\Lambda)}\leq C\,m_N^s\,(\log N)^d.
\end{split}
\end{equation}
Moreover, if $c_N\leq C\epsilon_N\,m_N^2\,(\log N)^d\leq \hat C$, then for all
$p\geq 1$, $\varphi\in L^p(\Lambda)$,
\begin{equation}\label{hh2}
\begin{split}
\epsilon_N\left\|\sum_{j,k=1}^d\del_j\del_k R_N^{j,k}(\cdot,t)\ast \varphi(\cdot)\right\|_{L^p(\Lambda)}\leq c_N\,\|\varphi\|_{L^p(\Lambda)}.
\end{split}
\end{equation}

\end{lemma}

\section{Spectrally small truncation error for symmetric $\mu$}\label{sec:trunc_err}
In this section we assume that the measure $\mu$ is symmetric, cf.~\eqref{symm}. In the SVV approximation \eqref{2}, the convection term $\partial_x\cdot f(u)$
is replaced by $\partial_x\cdot P_Nf(u_N)$ which leads to the (truncation) term error
\begin{equation*}
\begin{split}
\partial_x\cdot(I-P_N)f(u_N).
\end{split}
\end{equation*}
We will now show that this error is spectrally small due to the
presence of the spectral vanishing viscosity term.

Let us start by noting that a straightforward estimate leads to
\begin{equation*}
\begin{split}
\|\partial_x^\alp(I-P_N)f(u_N)\|_{L^2(\Lambda)}=\bigg(\sum_{j=1}^d\sum_{|\xi|>N}|\xi^\alp|^2|\widehat{f_j(u_N)}(\xi)|^2\bigg)^{\frac12}\leq\frac{\|\partial_x^{\alp+\beta}
  f(u_N)\|_{L^2(\Lambda)}}{N^{|\beta|}}
\end{split}
\end{equation*}
for all multi-indices $\alp,\beta$. Note that there is no divergence in this estimate, so $\del_x^\alp  f$ is a vector. By Theorem 7.1 in \cite{Tadmor_MAIN},
there is a constant $\mathcal{K}_s$  such that
\begin{equation}\label{T_2.2}
\begin{split}
\|\partial_x^s f(u_N)\|_{L^2(\Lambda)}\leq\mathcal{K}_s\|\partial_x^su_N\|_{L^2(\Lambda)}
\quad\text{for}\quad
\mathcal{K}_s\leq C\sum_{k=1}^s|f|_{C^k}\|u_N\|_{L^\infty(\Lambda)}^{k-1}
\end{split}
\end{equation}
and $s=1,2,\dots$, where $|f|_{C^k}=\|\del_x^k f(\cdot)\|_{L^\infty(\Omega_N)}$ and
$\Omega_N=\{u:|u|\leq\|u_N\|_{L^\infty(\Lambda)}\}$. This inequality is
 a type of Gagliardo-Nirenberg-Moser estimate, and similar
results can be found in page 22 in Taylor \cite{Taylor}. By these two inequalities we can conclude that, for all $0\leq r\leq s$,
\begin{equation}\label{T_2.3}
\begin{split}
\|\partial_x^r(I-P_N)f(u_N)\|_{L^2(\Lambda)}\leq\frac{\mathcal{K}_s}{N^{s-r}}\,\|\partial_x^su_N\|_{L^2(\Lambda)}.
\end{split}
\end{equation}

Inequality \eqref{T_2.3} states that the $r$-derivative of the truncation error decays as rapidly as the $s$-smoothness of $u_N$ permits. Of course the
$s$-derivatives of an arbitrary $N$-trigonometric polynomial $u_N$ may grow as fast as $N^s$, in which case nothing is gained from \eqref{T_2.3}. However, if
$u_N$ is solves our VVS approximation \eqref{2},  we can have the better bound  $ \epsilon_N^{-s}$ in $L^2$. This will be a consequence of the following energy
estimate:
\begin{theorem}\label{Th:2.1}
Consider the SVV approximation \eqref{2} with  $\epsilon_N$
and $m_N$ such that
\begin{equation}\label{ass}\tag{\textbf{A}.7}
\left\{
\begin{split}
&\epsilon_N>\frac{8\,d^{\frac s2}\mathcal{K}_{s+1}}{N},\\
&\epsilon_N\, m_N^2\,(\log N)^d\leq C.
\end{split}
\right.
\end{equation}
Then there is a constant $\mathcal{B}_s$ (proportional to $\Pi_{k=1}^s\mathcal{K}_s$ for $s\geq 1$ and to $\|u_N\|_{L^\infty}$ for
$s=0$) such that
\begin{equation}\label{main_est}
\begin{split}
&\epsilon^s_N\|\partial_x^su_N(\cdot,t)\|_{L^2(\Lambda)}+\epsilon^s_N\left(-\sum_{|\alp|=s}\sum_{|\xi|\leq N}G^\mu(\xi)|\xi^\alp|^{2}\int_0^t|\hat
u_\xi(\tau)|^{2}\, \dif \tau\right)^{\frac{1}{2}}\\
&\qquad\qquad\qquad+\epsilon_N^{s+\frac{1}{2}}\|\partial_x^{s+1}u_N\|_{L^2(D_T)}\leq \mathcal{B}_s+ 3\epsilon_N^s\|\partial_x^su_N(\cdot,0)\|_{L^2(\Lambda)}.
\end{split}
\end{equation}
\end{theorem}

Remember that in this section $\mu$ is symmetric and hence $G^\mu$ is
real and non-positive. Now if
\medskip
\begin{itemize}
\item[(\textbf{A}.8)] $\quad|f|_{C^s}<\infty$ for sufficiently large $s$,
  cf.~\eqref{smooth_req} below, and
\bigskip
\item[(\textbf{A}.9)]  $\quad u_0$ is such that
  $\epsilon_N^s\,\|\partial_x^su_N(\cdot,0)\|_{L^2(\Lambda)}\leq C$,
\end{itemize}
\medskip
then Theorem \ref{Th:2.1} implies that
\begin{equation*}
\begin{split}
\|\partial_x^su_N(\cdot,t)\|_{L^2(\Lambda)}\leq C\,
\epsilon_N^{-s}\quad\text{and}\quad \|\partial_x^{s+1}u_N\|_{L^2(D_T)}&\leq C\, \epsilon_N^{-(s+\frac{1}{2})}.
\end{split}
\end{equation*}
Taking into account \eqref{T_2.3}, we then find that
\begin{align}
\|\partial_x^r(I-P_N)f(u_N(\cdot,t))\|_{L^2(\Lambda)}&\leq
C\,\mathcal{B}_s\,N^{-s_r},\quad s_r=s(1-\theta)-r,\label{k2}\\
\|\partial_x^r(I-P_N)f(u_N)\|_{L^2(D_T)}&\leq C\,\mathcal{B}_s\,N^{-(s_r+\frac{\theta}{2})},\quad \forall\, s\geq1.\label{k1}
\end{align}

We can now turn these inequalities into spectral decay estimates in the uniform norm using the Sobolev inequality (cf.~Theorem 6, Chapter 5, in \cite{Evans})
\begin{equation*}
\begin{split}
\|\partial_x^r\varphi\|_{L^\infty}\leq C\,\|\partial_x^{r+[\frac{d}{2}]+1}\varphi\|_{L^2}.
\end{split}
\end{equation*}
For example, inequality \eqref{k1} becomes
\begin{equation}\label{k3}
\begin{split}
\|\partial_x^r(I-P_N)f(u_N)\|_{L^\infty(D_T)}&\leq C\,\mathcal{B}_s\,N^{-s_r+[\frac{d}{2}]+1-\frac{\theta}{2}}\leq C\,\mathcal{B}_s\,N^{-s_r+[\frac{d}{2}]+1}.
\end{split}
\end{equation}
Note that the polynomial decay rate in \eqref{k3} can be made as large
as the $C^s$-smoothness of $f(\cdot)$ permits. Taking $r=2$, we can
find the following result.
\begin{theorem}
\label{thm2}
If $f\in C^s$ with
\begin{equation}\label{smooth_req}
\begin{split}
s\geq\frac{4+[\frac{d}{2}]}{1-\theta},
\end{split}
\end{equation}
then
\begin{equation}\label{lo}
\begin{split}
\|\partial_x(I-P_N)f(u_N)\|_{L^\infty(D_T)}+\|\partial_x^2(I-P_N)f(u_N)\|_{L^\infty(D_T)}\leq
\frac{C\,\mathcal{B}_s}{N}.
\end{split}
\end{equation}
\end{theorem}
The smoothness requirement \eqref{smooth_req} will be sufficient for
all the estimates derived throughout the  paper.

\begin{proof}[Proof of Theorem \ref{Th:2.1}]
For sake of brevity, we will write $\|\cdot\|$ instead of
$\|\cdot\|_{L^2(\Lambda)}$.
With \eqref{full_lap} in mind, we rewrite the SVV approximation
\eqref{2} in the two equivalent forms
\begin{align}\label{f1}
\begin{split}
\partial_t u_N+\partial_x\cdot P_Nf(u_N)-\Levy^\mu[u_N]-\epsilon_N\Delta u_N=-\epsilon_N\sum_{j,k=1}^d\del_j\del_kR^{j,k}_N\ast u_N,
\end{split}\\
\label{T_7.1}
\begin{split}
&\partial_t u_N+\partial_x\cdot f(u_N)-\Levy^\mu[u_N]-\epsilon_N\Delta u_N\\
&\qquad\qquad\qquad\qquad=-\epsilon_N\sum_{j,k=1}^d\del_j\del_kR^{j,k}_N\ast u_N+\partial_x\cdot(I-P_N)f(u_N).
\end{split}
\end{align}
Since $G^\mu(\xi)\leq0$ ($\mu$ is symmetric) and $u_N(x)$ and
$\Levy^\mu[u_N]$ are real,
$$\int_\Lambda \Levy^\mu[u_N]u_N\ \dif x=\sum_{|\xi|\leq N}G^\mu(\xi)|\hat u_\xi(t)|^2\leq 0,$$
and hence spatial integration of \eqref{T_7.1} against $u_N$ yields
\begin{equation*}
\begin{split}
&\frac{1}{2}\frac{\dif}{\dif t}\|u_N\|^2-\sum_{|\xi|\leq N}G^\mu(\xi)|\hat u_\xi(t)|^2+\epsilon_N\|\partial_xu_N\|^2\\
&\leq\epsilon_N\|u_N\|\bigg\|\sum_{j,k=1}^d\del_j\del_kR_N^{j,k}\ast u_N\bigg\|+\sum_{j=1}^d\|\partial_ju_N\|\|(I-P_N)f_j(u_N)\|.
\end{split}
\end{equation*}
Using \eqref{hh2} with $p=2$ for the first term on the right and \eqref{T_2.3} with $(r,s)=(0,1)$ for the second term, we find that
\begin{equation*}
\begin{split}
\frac{1}{2}\frac{\dif}{\dif t}\|u_N\|^2-\sum_{|\xi|\leq N}G^\mu(\xi)|\hat u_\xi(t)|^2+\left(\epsilon_N-\frac{\mathcal{K}_1}{N}\right)\|\partial_xu_N\|^2 \leq
c_N\|u_N\|^2
\end{split}
\end{equation*}
with $c_N\leq C\epsilon_N\,m_N^2\,(\log N)^d\leq \hat C$. Hence
\eqref{main_est} follows for $s=0$ since by \eqref{ass},
\begin{equation*}
\begin{split}
\left(\epsilon_N-\frac{\mathcal{K}_1}{N}\right)>\frac{\epsilon_N}{2},
\end{split}
\end{equation*}
and $c_N\|u_N\|^2\leq C\,\|u_N\|_{L^\infty(\Lambda)}^2=\mathcal{B}_0^2$.

The general case follows by induction on $s$. Spatial integration of \eqref{f1} against $\partial_x^{2\alp}u_N$ for some multi-index $\alp$ yields
\begin{equation}\label{f2}
\begin{split}&\frac{1}{2}\frac{\dif}{\dif t}\|\partial_x^\alp u_N\|^2-\sum_{|\xi|\leq N}G^\mu(\xi)|\xi^\alp|^{2}|\hat u_\xi(t)|^2+\epsilon_N\|\del^\alp_x\partial_xu_N\|^2\\
&\leq\epsilon_N\|\partial_x^{\alp}u_N\|\bigg\|\sum_{j,k=1}^d\del_j\del_kR_N^{j,k}\ast
\partial_x^{\alp}u_N\bigg\|+\|\partial_x^{\alp}\del_xu_N\|\|\partial_x^{|\alp|-1}\partial_x
\cdot P_Nf(u_N)\|.
\end{split}
\end{equation}
After having used \eqref{hh2} and Young's inequality to bound the first
and second term on the right hand side, we find that
\begin{equation*}
\begin{split}
&\frac12\frac{\dif}{\dif t}\|\partial_x^\alp u_N\|^2-\sum_{|\xi|\leq N}G^\mu(\xi)|\xi^\alp|^{2}|\hat u_\xi(t)|^2+\frac{\epsilon_N}{2}\|\del^\alp_x\partial_xu_N\|^2\\
&\qquad\qquad\qquad\qquad\qquad\leq C\,\|\partial_x^{\alp}u_N\|^2+\frac{1}{2\epsilon_N}\|\partial_x^{|\alp|} P_Nf(u_N)\|^2.
\end{split}
\end{equation*}
Now we sum over all $|\alp|= s$ to find that
\begin{equation}\label{f4}
\begin{split}
&\frac12\frac{\dif}{\dif t}\|\partial_x^s u_N\|^2-\sum_{|\alp|= s}\sum_{|\xi|\leq N}G^\mu(\xi)|\xi^\alp|^{2}|\hat u_\xi(t)|^2+\frac{\epsilon_N}{2}\|\del^{s+1}_xu_N\|^2\\
&\qquad\qquad\qquad\qquad\qquad\leq C\,\|\partial_x^{s}u_N\|^2+\frac{d^s}{2\epsilon_N}\|\partial_x^{s} P_Nf(u_N)\|^2.
\end{split}
\end{equation}
By \eqref{T_2.2} and \eqref{T_2.3},
\begin{equation*}
\begin{split}
\|\partial_x^{s} P_Nf(u_N)\|&\leq \|\partial_x^{s} f(u_N)\|+\|\partial_x^{s}(I-P_N) f(u_N)\|\\
&\leq \mathcal{K}_{s}\,\|\partial_x^{s}u_N\|+\frac{\mathcal{K}_{s+1}}{N}\,\|\partial_x^{s+1}u_N\|,
\end{split}
\end{equation*}
and hence by inequality \eqref{f4} we see that
\begin{equation}\label{pp2}
\begin{split}
\frac12\frac{\dif}{\dif t}\|\partial_x^s u_N\|^2-\sum_{|\alp|= s}\sum_{|\xi|\leq N}G^\mu(\xi)|\xi^\alp|^{2}|\hat u_\xi(t)|^2+\left(\frac{\epsilon_N}{2}
-\frac{d^s\mathcal{K}^2_{s+1}}{N^2\epsilon_N}\right)\|\del_x^{s+1}u_N\|^2\\
\qquad\qquad\qquad\qquad\qquad \leq
\left(C+\frac{d^s\mathcal{K}_s^2}{\epsilon_N}\right)\,\|\partial_x^{s}u_N\|^2\leq\frac{2\,d^s\mathcal{K}_s^2}{\epsilon_N}\,\|\partial_x^{s}u_N\|^2,
\end{split}
\end{equation}
where the last inequality holds for $N$ big enough. By \eqref{ass} and
integration in time,  we then find that
\begin{equation}\label{pp3}
\begin{split}
\frac12\|\partial_x^su_N(\cdot,t)\|^2-\sum_{|\alp|= s}\sum_{|\xi|\leq N}G^\mu(\xi)|\xi^\alp|^{2}\int_0^t|\hat u_\xi(\tau)|^2\, \dif \tau +\frac{\epsilon_N}{4}
\|\partial_x^{s+1}u_N\|^2_{L^2(D_T)}\\
\leq \frac{2d^s\mathcal{K}_s^2}{\epsilon_N}\,\|\partial_x^{s}u_N\|^2_{L^2(D_T)}+\frac12\|\partial_x^su_N(\cdot,0)\|^2.
\end{split}
\end{equation}
At this point \eqref{main_est} follows by the induction assumption on $s$ since
\begin{equation*}
\begin{split}
\|\partial_x^{s}u_N\|^2_{L^2(D_T)}\leq C\,\mathcal{B}^2_{s-1} \epsilon_N^{-(2s-1)}.
\end{split}
\end{equation*}
The proof is now complete.
\end{proof}

\section{A priori estimates and compactness}\label{sec:aprior}
In this section we prove uniform
\begin{equation*}
\begin{split}
\text{$L^\infty(D_T)$, $L^\infty(0,T;BV(\Lambda))$, and $C^{0,\frac12}([0,T];L^1(\Lambda))$}
\end{split}
\end{equation*}
bounds on the solutions $\{u_N:N\in\mathbb{N}\}$ of the SVV approximation \eqref{2}. As a consequence we obtain compactness in $L^1$.

\subsection{Regularity in space}

\begin{lemma}\label{lem:inf}\emph{($L^\infty$-stability)}
Let $(\mathbf{A}.1)$--$(\mathbf{A}.9)$ and \eqref{symm} hold and $u_N$ be the solution
of the SVV approximation \eqref{2}. Then for
$t<C\ln N$,
\begin{equation*}
\begin{split}
\|u_N(\cdot,t)\|_{L^\infty(\Lambda)}\leq C\,\|u_N(\cdot,0)\|_{L^\infty(\Lambda)}.
\end{split}
\end{equation*}
\end{lemma}

\begin{proof}
For sake of brevity, we write just $\|\cdot\|_{\infty}$ instead of $\|\cdot\|_{L^\infty(\Lambda)}$.
First we  note that, for any smooth convex function $\eta(\cdot)$ with derivative $\eta'(\cdot)$, we have that
\begin{equation}\label{h1}
\begin{split}
\eta'(u_N)\,\Levy^\mu[u_N]\leq\Levy^\mu[\eta(u_N)].
\end{split}
\end{equation}
This is a consequence of the inequality $\eta'(b)(a-b)\leq\eta(a)-\eta(b)$ which holds for all smooth convex functions $\eta(\cdot)$. Moreover,
\begin{equation}\label{h2}
\begin{split}
\int_{\Lambda} \Levy^\mu[\eta(u_N(\cdot,t))](x)\ \dif x=0.
\end{split}
\end{equation}
To see this note that
\begin{equation*}
\begin{split}
\int_{\Lambda}\int_{|z|>0}\left|\eta(u_N(x+z))-\eta(u_N(x))+z\cdot\partial_x\eta(u_N(x))\,\mathbf{1}_{|z|<1}\right|\ \dif\mu(z)\,\dif x\\
\leq\|\partial^2_x \eta(u_N)\|_\infty\int_{|z|<1}|z|^2\ \dif\mu(z)+\|\eta(u_N)\|_\infty\int_{|z|>1}\dif\mu(z)<\infty,
\end{split}
\end{equation*}
since $u_N$ is smooth and periodic. By Fubini we then find that
\begin{equation*}
\begin{split}
&\int_{\Lambda} \Levy^\mu[\eta(u_N(\cdot,t))](x)\ \dif x\\
&=\int_{|z|>0}\int_{\Lambda}\eta(u_N(x+z))-\eta(u_N(x))+z\cdot\partial_x\eta(u_N(x))\,\mathbf{1}_{|z|<1}\ \dif x\,\dif\mu(z).
\end{split}
\end{equation*}
By $\Lambda$-periodicity of $u_N$, \eqref{h2} now follows since
\begin{equation*}
\int_{\Lambda}\eta(u_N(x+z))\,\dif x=\int_{\Lambda}\eta(u_N(x))\, \dif x
\end{equation*}
for every $z$, and
\begin{align*}
&\int_{\Lambda}\del_{x_i}\eta(u_N(x',x_i))\, \dif x'\dif x_i\\
& = \int_{(0,2\pi)^{d-1}}\eta(u_N(x',2\pi))\, \dif x'-\int_{(0,2\pi)^{d-1}}\eta(u_N(x',0))\, \dif x'=0.
\end{align*}

Let us now integrate \eqref{T_7.1} against the function $p\, u_N^{p-1}$ (with $p$ even), and use \eqref{h1} and \eqref{h2} to get rid of the non-local operator
$\Levy^\mu[\cdot]$. We then find that
\begin{equation*}
\begin{split}
&p\,\|u_N(\cdot,t)\|^{p-1}_{L^p(\Lambda)}\,\frac{\dif}{\dif
  t}\|u_N(\cdot,t)\|_{L^p(\Lambda)}=\frac{\dif}{\dif
  t}\|u_N(\cdot,t)\|_{L^p(\Lambda)}^p=\int_\Lambda
u_N^{p-1}(x,t)\partial_tu_N(x,t) dx\\
&\quad\leq p\int_\Lambda u_N^{p-1}(x,t)\left(\epsilon_N\sum_{j,k=1}^d\del_j\del_kR^{j,k}_N\ast u_N(x,t)+\partial_x\cdot(I-P_N)f(u_N(x,t))\right)\dif x
\end{split}
\end{equation*}
which by the H\"{o}lder inequality (with $p$ and $q=\frac p{p-1}$) is less than or equal to
\begin{equation*}
\begin{split}
&p\,\|u_N(\cdot,t)^{p-1}\|_{L^{\frac p{p-1}}(\Lambda)}\\
&\qquad\left(\epsilon_N\bigg\|\sum_{j,k=1}^d\del_j\del_kR^{j,k}_N\ast u_N(\cdot,t)\bigg\|_{L^p(\Lambda)}
+\|\partial_x\cdot(I-P_N)f(u_N(\cdot,t))\|_{L^p(\Lambda)}\right).
\end{split}
\end{equation*}
Since $\|\phi^{p-1}\|_{L^{\frac p{p-1}}}=\|\phi\|_{L^p}^{p-1}$ , we may
divide both sides by $p\,\|u_N(\cdot,t)\|^{p-1}_{L^p(\Lambda)}$ and
send $p\rightarrow\infty$ to discover that
\begin{equation*}
\begin{split}
\frac{\dif}{\dif t}\|u_N(\cdot,t)\|_{\infty}\leq\epsilon_N\bigg\|\sum_{j,k=1}^d\del_j\del_kR^{j,k}_N\ast
u_N(\cdot,t)\bigg\|_{\infty}+\|\partial_x\cdot(I-P_N)f(u_N(\cdot,t))\|_{\infty}.
\end{split}
\end{equation*}
By \eqref{lo}, \eqref{hh2}, the definitions of $\mathcal{B}_s$,
$\mathcal{K}_s$ and $c_N$, and ({\bf A}.7), it follows that
\begin{align*}
&\|\partial_x\cdot(I-P_N)f(u_N(\cdot,t))\|_{\infty}\leq\frac{\mathcal{B}_s}{N}
\leq\frac C N \prod_{k=1}^s\mathcal{K}_s\leq \frac{\hat C}N\|u_N\|_{\infty}^{\frac{s^2}{2}},\\
&\epsilon_N\bigg\|\sum_{j,k=1}^d\del_j\del_kR^{j,k}_N(\cdot,t)\ast
  u_N(\cdot,t)\bigg\|_{\infty}\leq c_N\|u_N\|_{\infty}\leq C\|u_N\|_{\infty},
\end{align*}
and hence
\begin{equation*}
\begin{split}
\frac{\dif}{\dif t}\|u_N(\cdot,t)\|_{\infty}\leq c_N\|u_N(\cdot,t)\|_{\infty}+\frac{C}{N}\,\|u_N(\cdot,t)\|_{\infty}^{\frac{s^2}{2}}.
\end{split}
\end{equation*}
Letting $y(t)=e^{-c_Nt}\|u_N(\cdot,t)\|_{\infty}$, and multiplying by the
integrating factor $e^{-c_Nt}$, we find that
\begin{equation*}
\begin{split}
\frac{\dif y}{\dif t}(t)\leq
\frac{C}{N}\,y^\frac{s^2}{2}(t)\,e^{c_N(\frac{s^2}{2}-1)t}
\end{split}
\end{equation*}
which implies that
\begin{equation*}
\begin{split}
y(t)\leq y(0)\left(1-\frac{C\Big(e^{c_N(\frac{s^2}{2}-1)t}-1\Big)y^{\frac{s^2}{2}-1}(0)}{Nc_N}\right)^{-\frac1{\frac{s^2}2-1}}.
\end{split}
\end{equation*}
Going back to $\|u_N(\cdot,t)\|_{\infty}$, we can conclude that
\begin{equation*}
\begin{split}
\|u_N(\cdot,t)\|_{\infty}\leq e^{c_N
t}\|u_N(\cdot,0)\|_{\infty}\left(1-\frac{Ce^{c_N(\frac{s^2}{2}-1)t}\|u_0\|_\infty^{\frac{s^2}{2}-1}}{Nc_N}\right)^{-\frac{2}{2-s^2}},
\end{split}
\end{equation*}
where the last factor is bounded for $t\leq
C\ln N$ for some $C$.
\end{proof}

We also have the following result:

\begin{lemma}\label{lem:BV}\emph{($BV$-stability)}
Let  $(\mathbf{A}.1)$--$(\mathbf{A}.9)$ and \eqref{symm} hold, and $u_N$ be the solution
of the SVV  approximation \eqref{2}. Then
\begin{equation*}
\begin{split}
\|u_N(\cdot,T)\|_{BV(\Lambda)}\leq e^{c_NT}\Bigg(\|u_N(\cdot,0)\|_{BV(\Lambda)}+C\,N^{-s_2}\Bigg)
\end{split}
\end{equation*}
with $c_N=\epsilon_N\,m_N^2\,(\log N)^d\leq C$ and $s_2=s(1-\theta)-2>0$.
\end{lemma}

\begin{proof}
Spatial differentiation of \eqref{T_7.1} yields
\begin{equation*}
\begin{split}
&\partial_t\partial_iu_N+\partial_x\cdot(f'(u_N)\partial_iu_N)-\Levy^\mu[\partial_iu_N]-\epsilon_N\,\Delta\partial_iu_N\\
&\qquad\qquad\qquad=\partial_i\partial_x\cdot(I-P_N)f(u_N)+\epsilon_N\sum_{j,k=1}^d\del_j\del_kR_N^{j,k}\ast\partial_iu_N.
\end{split}
\end{equation*}
If we integrate this expression against $\mathrm{sgn}_\varrho(\partial_iu_N)$, where $\mathrm{sgn}_\varrho(\cdot)$ is a smooth approximation of the sign
function, we can get rid of the non-local operator $\Levy^\mu[\cdot]$ as in the proof of Lemma \ref{lem:inf}. If we also use \eqref{hh2} with $p=1$ and take
the limit as $\varrho\ra0$, a standard computations reveal that
\begin{equation*}
\begin{split}
\frac{\dif}{\dif t}\|\partial_iu_N(\cdot,t)\|_{L^1(\Lambda)}\leq
C\,\|\del_i\partial_x\cdot(I-P_N)f(u_N)\|_{L^1(\Lambda)}+c_N\|\partial_iu_N(\cdot,t)\|_{L^1(\Lambda)}.
\end{split}
\end{equation*}
Since
$\|u_N(\cdot,t)\|_{BV(\Lambda)}\leq\sum_{i=1}^d\|\partial_iu_N(\cdot,t)\|_{L^1(\Lambda)}$,
we integrate this inequality in time to see that
\begin{equation*}
\begin{split}
\|u_N(\cdot,t)\|_{BV(\Lambda)}\leq e^{c_Nt}\Bigg(\|u_N(\cdot,0)\|_{BV(\Lambda)}+C\,\|\partial_x^2(I-P_N)f(u_N)\|_{L^1(D_T)}\Bigg).
\end{split}
\end{equation*}
But by \eqref{k2},
\begin{equation*}
\begin{split}
\|\partial_x^2(I-P_N)f(u_N)\|_{L^1(D_T)}&\leq C\,\|\partial_x^2(I-P_N)f(u_N)\|_{L^2(D_T)}\leq C\,\mathcal{B}_s\,\sqrt T\, N^{-s_2},
\end{split}
\end{equation*}
and the proof is complete.
\end{proof}

\subsection{Regularity in time}

\begin{lemma}\label{lem:time_reg}\emph{(Regularity in time)}
Let $(\mathbf{A}.1)$--$(\mathbf{A}.9)$ and \eqref{symm} hold and,
$u_N$ be the solution
of the SVV approximation \eqref{2}. Then
\begin{equation*}
\begin{split}
\|u_N(\cdot,t_1)-u_N(\cdot,t_2)\|_{L^1(\Lambda)}\leq C\,\sqrt{|t_1-t_2|}.
\end{split}
\end{equation*}
\end{lemma}

\begin{proof}
Let $u_N^\epsilon(\cdot,t)=u_N(\cdot,t)\ast \omega_\epsilon(\cdot)$
for an approximate unit $\omega_\epsilon$ (cf.~the proof of Theorem
\ref{th:uniqueness}). By the triangle inequality we see that
\begin{equation}\label{ooo}
\begin{split}
&\|u_N(\cdot,t_1)-u_N(\cdot,t_2)\|_{L^1(\Lambda)}\leq
\|u_N(\cdot,t_1)-u_N^\epsilon(\cdot,t_1)\|_{L^1(\Lambda)}\\
&\quad+\|u_N^\epsilon(\cdot,t_1)-u_N^\epsilon(\cdot,t_2)\|_{L^1(\Lambda)} +\|u_N^\epsilon(\cdot,t_2)-u_N(\cdot,t_2)\|_{L^1(\Lambda)}.
\end{split}
\end{equation}
The first and the third term on the right-hand side of \eqref{ooo} are
bounded by $\epsilon|u|_{BV}$:
\begin{equation*}
\begin{split}
\|u_N(\cdot,t)-u^\epsilon_N(\cdot,t)\|_{L^1(\Lambda)}&=\int_\Lambda\left|\int_{\R^d}\omega_\epsilon(y-x)\Big(u_N(x,t)-u_N(y,t)\Big)\, \dif y\right|\dif x\\
&\leq\int_\Lambda\int_{\R^d}\omega_\epsilon(s)\Big|u_N(x,t)-u_N(s+x,t)\Big|\, \dif s\, \dif x\\
&\leq \sqrt{d}\,|u(\cdot,t)|_{BV(\Lambda)}\int_{\R^d}|s|\,\omega_\epsilon(s)\, \dif s\\
&\leq \sqrt{d}\,\epsilon\,|u(\cdot,t)|_{BV(\Lambda)}.
\end{split}
\end{equation*}
Let us estimate the second term. By Taylor's formula with integral
remainder,
\begin{equation*}
\begin{split}
&\|u_N^\epsilon(\cdot,t_1)-u_N^\epsilon(\cdot,t_2)\|_{L^1(\Lambda)}\\
&\qquad\qquad\leq|t_1-t_2|\int_\Lambda\int_0^1|\partial_t u_N^\epsilon(x,t_1+\tau(t_2-t_1))|\ \dif \tau\,\dif x.
\end{split}
\end{equation*}
We now derive a bound for $\|\partial_t u_N\|_{L^1}$ (and hence also
for $\|\partial_t u_N^\eps\|_{L^1}$) by using the SVV approximation
\eqref{2} itself. To this end, we take the convolution product of
both sides of \eqref{f1} with $\omega_\epsilon$ to obtain
\begin{equation*}
\begin{split}
\|\partial_t u_N^\epsilon\|_{L^1(\Lambda)}&\leq \|\partial_x\cdot P_Nf(u_N)\ast \omega_\epsilon\|_{L^1(\Lambda)}+\|\Levy^\mu[u_N]\ast
\omega_\epsilon\|_{L^1(\Lambda)}\\
&\qquad+\epsilon_N\,\|\Delta u_N\ast\omega_\epsilon\|_{L^1(\Lambda)}+\epsilon_N\left\|\left(\sum_{j,k=1}^d\del_j\del_kR_N^{j,k}\ast u_N\right)\ast
\omega_\epsilon\right\|_{L^1(\Lambda)}\\
&=I_1+I_2+I_3+I_4.
\end{split}
\end{equation*}
By the triangle inequality and Young's inequality for convolutions,
\begin{equation*}
\begin{split}
I_1&=\|\partial_x\cdot P_Nf(u_N)\ast \omega_\epsilon\|_{L^1(\Lambda)}\\
&\leq \|\partial_x\cdot f(u_N)\ast \omega_\epsilon\|_{L^1(\Lambda)}+\|\partial_x\cdot (I-P_N)f(u_N)\ast \omega_\epsilon\|_{L^1(\Lambda)}\\
&\leq \|\partial_x\cdot f(u_N)\|_{L^1(\Lambda)}+\|\partial_x\cdot (I-P_N)f(u_N)\|_{L^1(\Lambda)}.
\end{split}
\end{equation*}
Therefore, by the regularity of $f$ and $u_N$ (({\bf A}.8), Lemmas
\ref{lem:inf}-\ref{lem:BV}) and \eqref{lo}, we find that
\begin{equation*}
\begin{split}
I_1\leq C\left(|u(\cdot,t)|_{BV(\Lambda)}+\frac{1}{N}\right).
\end{split}
\end{equation*}
For the term containing the non-local operator we write
\begin{equation*}
\begin{split}
&I_2\leq\int_\Lambda\left|\left(\int_{\R^d}\int_{|z|<1}u_N(x+z)-u_N(x)-z\cdot\partial_xu_N(x)\, \dif \mu(z)\right)\omega_\epsilon(x-y)\ \dif y\right|\dif x\\
&\qquad+\int_\Lambda\left|\left(\int_{\R^d}\int_{|z|>1}u_N(x+z)-u_N(x)\, \dif \mu(z)\right)\omega_\epsilon(x-y)\ \dif y\right|\dif x
\end{split}
\end{equation*}
The second term on the right hand side of the inequality above is easily seen to be bounded by $C\|u_N(\cdot,t)\|_{L^1}$, while Taylor's formula with integral
reminder and integration by parts reveals that the first term is bounded by
\begin{equation*}
\begin{split}
&\int_\Lambda\int_{\R^d}\int_{|z|<1}\int_0^1(1-\tau)\,|z|^2\,|\partial_xu_N(x,t)|\,|\partial_x\omega_\epsilon(x-y)|\ \dif \tau\,\dif \mu(z)\, \dif y\,\dif
x\\
&\qquad\qquad\qquad\qquad\qquad\qquad\qquad\qquad\qquad\qquad \leq C\,\epsilon^{-1}\,|u|_{BV(\Lambda)}.
\end{split}
\end{equation*}
 For the Laplace term we have
\begin{equation*}
\begin{split}
I_3\leq \|\partial_xu*\partial_x\omega_\epsilon\|_{L^1(\Lambda)}\leq \epsilon^{-1}\,|u|_{BV(\Lambda)},
\end{split}
\end{equation*}
and finally, using Young's inequality for convolutions and \eqref{hh2},
\begin{equation*}
\begin{split}
I_4=\epsilon_N\left\|\left(\sum_{j,k=1}^d\del_j\del_kR_N^{j,k}\ast u_N\right)\ast \omega_\epsilon\right\|_{L^1(\Lambda)}\leq C\,\|u_N\|_{L^1(\Lambda)}.
\end{split}
\end{equation*}
To sum up we have
\begin{equation*}
\begin{split}
\|\partial_t u_N^\epsilon\|_{L^1(\Lambda)}\leq\|\partial_t u_N\|_{L^1(\Lambda)}\leq C\,\left(1+\frac{1}{\epsilon}\right),
\end{split}
\end{equation*}
and inequality \eqref{ooo} and the above estimates then implies that
\begin{equation*}
\begin{split}
&\|u_N(\cdot,t_1)-u_N(\cdot,t_2)\|_{L^1(\Lambda)}\leq C\bigg(\epsilon+|t_1-t_2|\,\left(1+\epsilon^{-1}\right)\bigg).
\end{split}
\end{equation*}
Take $\eps=\sqrt{|t_1-t_2|}$ and the proof is complete.
\end{proof}

\subsection{Compactness}
Thanks to the space/time a priori estimates in Lemmas \ref{lem:inf} -- \ref{lem:time_reg} and a Helly like compactness theorem, cf.~Theorem A.8 in
\cite{Holden/Risebro}, the family $\{u_N:N\in\mathbb{N}\}$ of solutions of the SVV approximation \eqref{2} is compact.

\begin{theorem}[Compactness]
Let $(\mathbf{A}.1)$--$(\mathbf{A}.9)$ and \eqref{symm} hold, and $u_N$ be the solution of the SVV approximation \eqref{2}. Then there exists a subsequence $u_N$ converging in
$C([0,T];L^1(\Lambda))$ to a limit $u\in C([0,T];L^1(\Lambda))\cap L^\infty(D_T)\cap L^\infty(0,T;BV(\Lambda))$.
\end{theorem}

\section{Convergence and error estimate}\label{sec:conv_anal}

The solution $v_{\epsilon_N}$ of the vanishing
viscosity method \eqref{prob:viscous} converges to the unique entropy solution $u$ of
\eqref{1}, and by  Theorem \ref{lemma_rate},
$$\|u(\cdot,t)-v_{\epsilon_N}(\cdot,t)\|_{L^1(\Lambda)}\leq C\,\sqrt{\epsilon_N}.$$
In this section we prove a similar error estimate between
$v_{\epsilon_N}$ and the SVV approximation
$u_N$.
\begin{theorem}\label{th:error}
Let $(\mathbf{A}.1)$--$(\mathbf{A}.9)$ and \eqref{symm} hold, $u_N$ be the solution of
the SVV method \eqref{prob:viscous}, and $v_{\epsilon_N}$ be the solution of
\eqref{2}. Then
\begin{equation*}
\begin{split}
\|u_{N}(\cdot,T)-v_{\epsilon_N}(\cdot,T)\|_{L^1(\Lambda)}\leq C\,\sqrt{\epsilon_N}.
\end{split}
\end{equation*}
\end{theorem}

A direct consequence of Theorems \ref{lemma_rate} and \ref{th:error}, is the
following convergence and error estimate for the SVV method.

\begin{corollary}\emph{(Convergence with rate)}
Let  $(\mathbf{A}.1)$--$(\mathbf{A}.9)$ and \eqref{symm} hold, $u_N$ be the solution of
the SVV method \eqref{2}, and $u$ be an entropy solution of \eqref{1}. Then
\begin{equation*}
\begin{split}
\|u(\cdot,T)-u_N(\cdot,T)\|_{L^1(\Lambda)}\leq C\,\sqrt{\epsilon_N}.
\end{split}
\end{equation*}
\end{corollary}

\begin{proof}[Proof of Theorem \ref{th:error}]
 Since $v_{\epsilon_N}$ is smooth, we can subtract equation
 \eqref{prob:viscous}  from equation \eqref{2} to obtain
\begin{equation*}
\begin{split}
\partial_t(u_N-v_{\epsilon_N})+\partial_x\cdot(f(u_N)-f(v_{\epsilon_N}))-\Levy^\mu[u_N-v_{\epsilon_N}]-\epsilon_N\Delta(u_N-v_{\epsilon_N})\\
\qquad\qquad\qquad\qquad\qquad\qquad=-\epsilon_N\sum_{j,k=1}^d\partial_jR_N^{j,k}\ast\partial_ku_N+\partial_x(I-P_N)f(u_N).
\end{split}
\end{equation*}
As explained in the proof of Lemma \ref{lem:BV}, we can integrate
such an inequality against (a smooth approximation of)
$\mathrm{sgn}(u_N-v_{\epsilon_N})$, to find that (after going to the limit)
\begin{equation*}
\begin{split}
&\frac{\dif}{\dif t}\|u_N-v_{\epsilon_N}\|_{L^1(\Lambda)}\\
&\leq\epsilon_N\Big\|\sum_{j,k=1}^d\partial_jR_N^{j,k}(\cdot,t)\ast\partial_ku_N(\cdot,t)\Big\|_{L^1(\Lambda)}
+\|\partial_x\cdot(I-P_N)f(u_N(\cdot,t))\|_{L^1(\Lambda)}.
\end{split}
\end{equation*}
By \eqref{hh1} with $r=s=1$, ({\bf A}.4),  ({\bf A}.5), and Lemma \ref{lem:BV},
\begin{equation*}
\begin{split}
&\Big\|\sum_{j,k=1}^d\partial_jR_N^{j,k}(\cdot,t)\ast\partial_ku_N(\cdot,t)\Big\|_{L^1(\Lambda)}\leq \Big\|\sum_{j,k=1}^d\partial_jR_N^{j,k}(\cdot,t)\Big\|_{L^1(\Lambda)}\Big\|\partial_ku_N(\cdot,t)\Big\|_{L^1(\Lambda)}\\
&\quad\leq C\, \,m_N\,(\log N)^d\|u_N(\cdot,t)\|_{BV(\Lambda)}\leq C\,\epsilon_N^{-\frac12},
\end{split}
\end{equation*}
so we can integrate in time to obtain
\begin{equation*}
\begin{split}
\|u_N(\cdot,t)-v_{\epsilon_N}(\cdot,t)\|_{L^1(\Lambda)}&\leq C\,\sqrt{\epsilon_N}+\|\partial_x\cdot(I-P_N)f(u_N(\cdot,T))\|_{L^1(D_T)}\\
&\leq C\bigg(\sqrt{\epsilon_N}+\|\partial_x\cdot(I-P_N)f(u_N(\cdot,T))\|_{L^2(D_T)}\bigg).
\end{split}
\end{equation*}
By \eqref{k1},
\begin{equation*}
\begin{split}
\|\partial_x\cdot(I-P_N)f(u_N(\cdot,T))\|_{L^2(D_T)}&\leq
C\,\mathcal{K}_s\,N^{-(s_1+\frac{\theta}{2})}\leq
C\,\mathcal{K}_s\,N^{-\frac{\theta}{2}}= C\,\sqrt{\epsilon_N},
\end{split}
\end{equation*}
since $s_1=s(1-\theta)-1>0$, cf.~\eqref{smooth_req}. The proof is now complete.
\end{proof}

\section{An application: the fractional Burgers' equation}\label{sec:num_ex}

In this section we apply the results of the previous sections to numerically solve the fractional (or fractal) Burgers' equation in $\R^d$,
\begin{equation}\label{frac_bur}
\left\{
\begin{array}{ll}
\partial_tu+u\sum_{j=1}^d\partial_{x_j} u=-(-\Delta)^{\lambda/2}u,&(x,t)\in D_T,\\
u(x,0)=u_{0}(x), &x\in\Lambda,
\end{array}
\right.
\end{equation}
where the fractional Laplacian term $-(-\Delta)^{\lambda/2}u_N=\Levy^{\pi_{\lambda}}[u_N]$ and $\pi_{\lambda}$ has been defined in \eqref{frac_lap}. In this
setting expression \eqref{weights} becomes
\begin{equation*}
\begin{split}
G^{\pi_{\lambda}}(\xi)=c_\lambda\int_{|z|>0}e^{i\xi\cdot
  z}-1-i\xi\cdot z\,\mathbf{1}_{|z|<1}\ \frac {\dif z}{|z|^{d+\lambda}},
\end{split}
\end{equation*}
with $c_\lambda=\lambda\,\Gamma(\frac{d+\lambda}{2})\left(2\pi^{\frac{d}{2}+\lambda}\,\Gamma(1-\frac{\lambda}{2})\right)^{-1}$,
cf.~\cite{Droniou/Imbert}.
We have the following result:

\begin{figure}[t]
\subfigure[$\lambda=1.6$]{
\includegraphics[width=60mm,height=30mm]{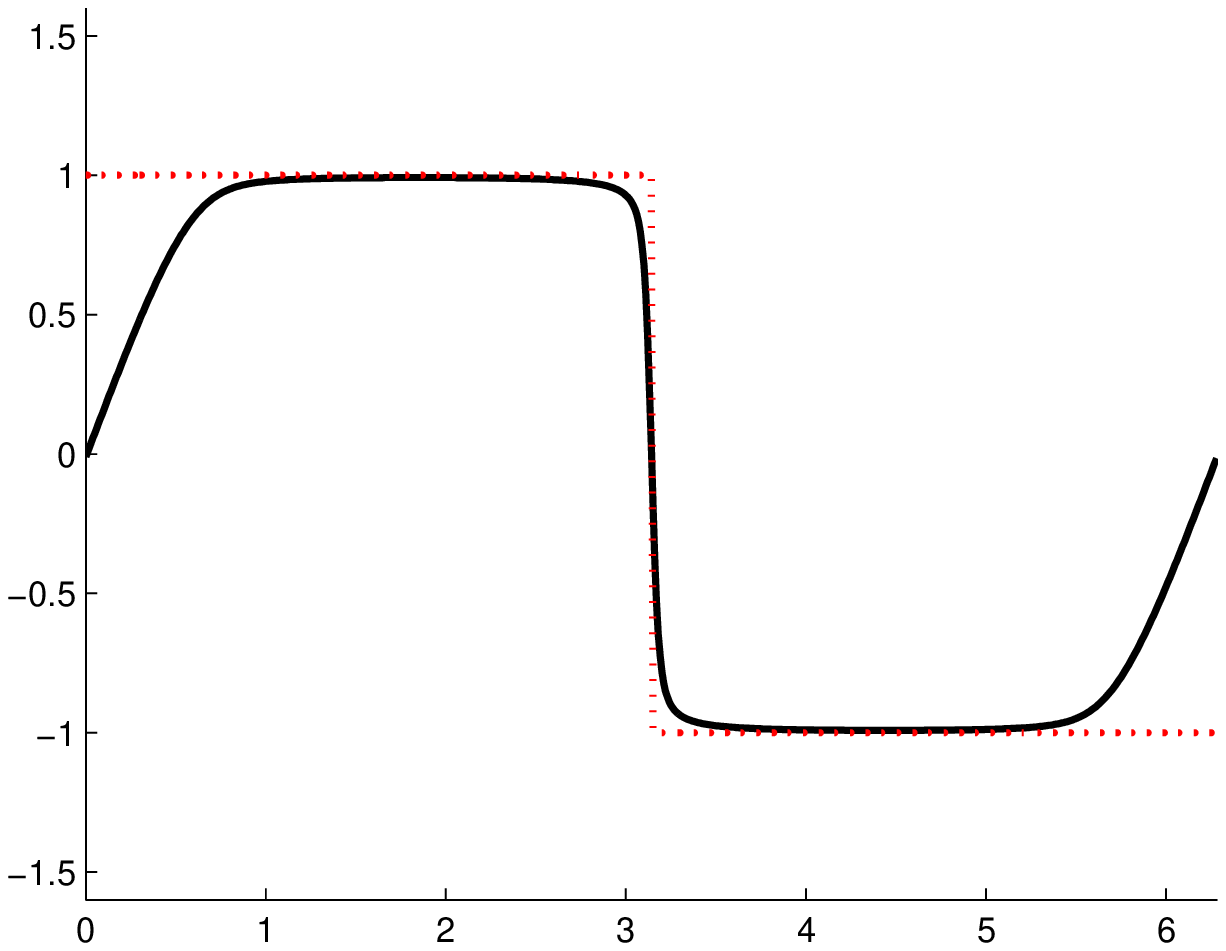}}
\subfigure[$\lambda=1.1$]{
\includegraphics[width=60mm,height=30mm]{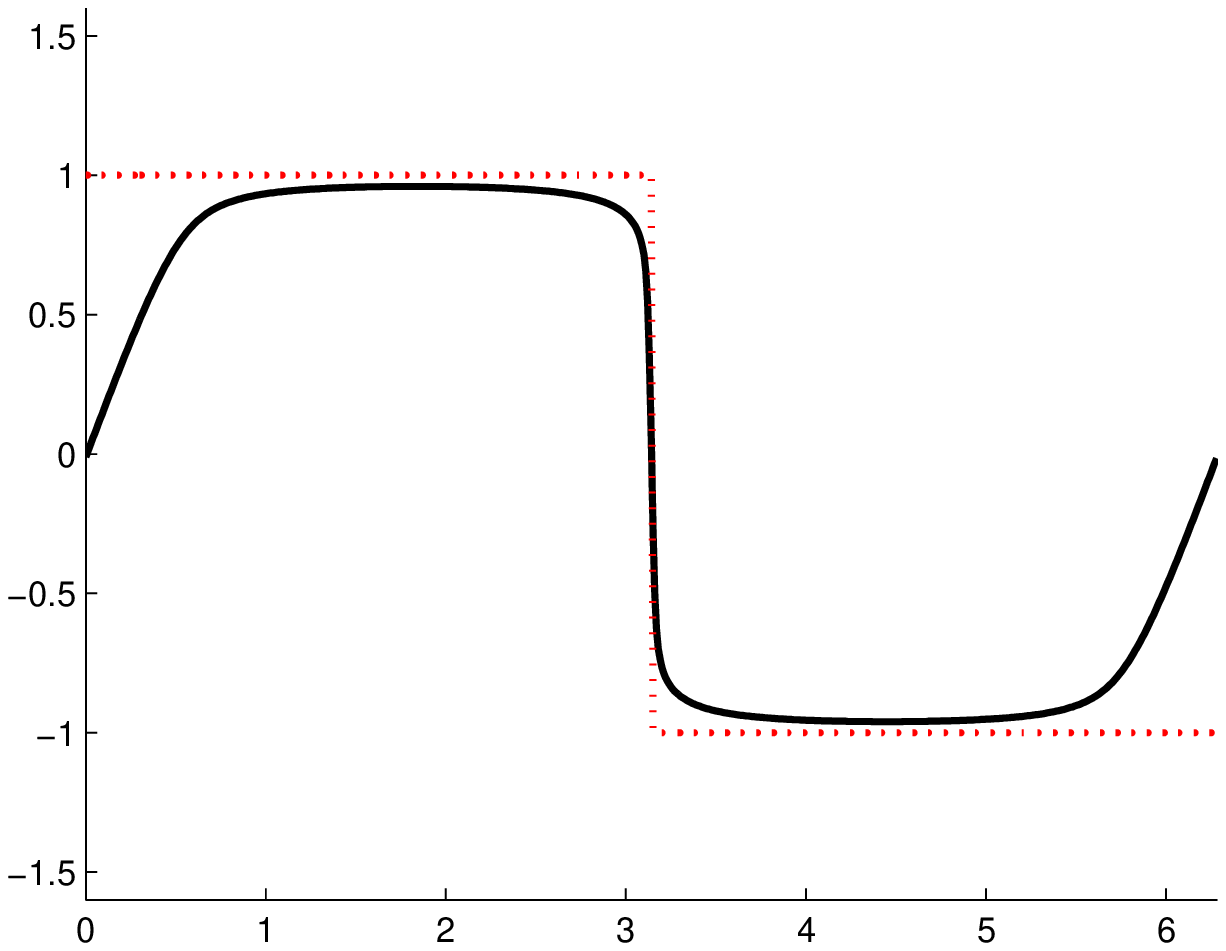}}\\
\subfigure[$\lambda=0.6$]{
\includegraphics[width=60mm,height=30mm]{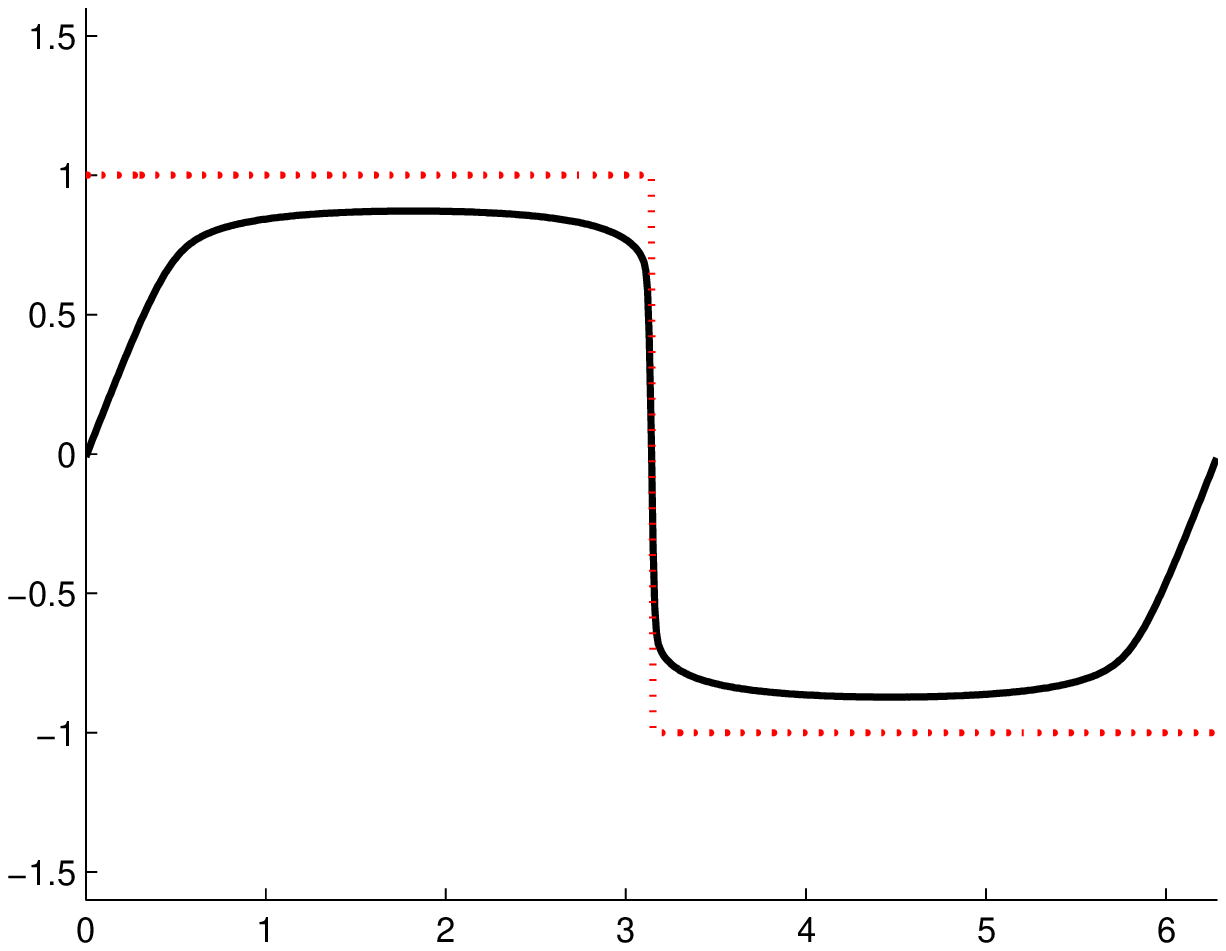}}
\subfigure[$\lambda=0.1$]{
\includegraphics[width=60mm,height=30mm]{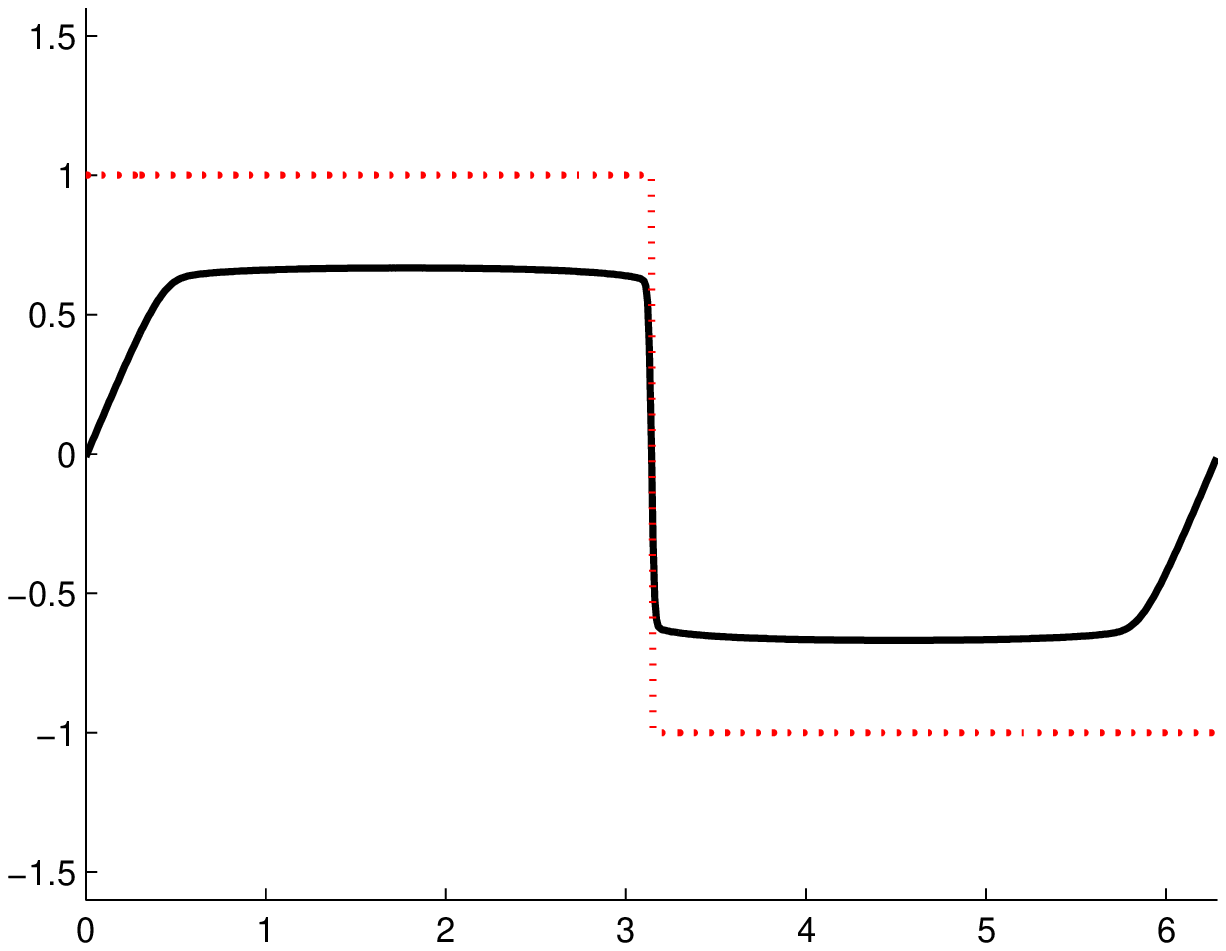}}
\caption{Solutions of system \eqref{3sys} with $N=256$ and $T=0.5$. The piecewise constant initial datum is $u_0(x)=\text{sgn}(\pi-x)$.}\label{fff1}
\end{figure}

\begin{proposition}\label{G}
\begin{equation}\label{weights_ll}
G^{\pi_{\lambda}}(\xi)=\left\{
\begin{split}
&-C_\lambda\,|\xi|^{\lambda}&\text{\emph{for} $d=1$,}\\
&-C_\lambda\,|\xi|^{\lambda}\int_{|y|=1}\dif S_y\,&\text{\emph{for} $d>1$,}
\end{split}
\right.
\end{equation}
where $C_\lambda=2\,c_\lambda\,\lambda^{-1}\int_0^\infty x^{-\lambda}\sin x\, \dif x>0$ and $\int_{|y|=1}\dif S_y=2\pi^{d/2}\,\Gamma^{-1}(\frac{d}{2})$.
\end{proposition}

The proof is given at the end of this section. In the above result and in the following, $dS_y$ will denote the surface area measure of the unit sphere
$|y|=1$. Expression \eqref{weights_ll} is the ``Fourier symbol'' of the fractional Laplace operator in our periodic setting. When $\lambda\in(0,1)$, the
integral $\Theta_\lambda=\int_0^\infty x^{-\lambda}\sin x\, \dif x$ is a \emph{generalized Fresnel integral} \cite{Loya} with value
\begin{equation*}
\begin{split}
\Theta_\lambda=\Gamma(1-\lambda)\sin\left(\frac{\pi(1-\lambda)}{2}\right).
\end{split}
\end{equation*}
When $\lambda=1$, $\Theta_\lambda$ is a \emph{Dirichlet integral}
\cite{Jeffreys} and has value $\frac{\pi}{2}$. For $\lambda\in(1,2)$,
the integral $\Theta_\lambda$ has to be evaluated numerically since explicit formulas are not available.

\begin{figure}[t]
\subfigure[$\lambda=1.6$]{
\includegraphics[width=60mm,height=30mm]{gig1.eps}}
\subfigure[$\lambda=1.1$]{
\includegraphics[width=60mm,height=30mm]{gig2.eps}}\\
\subfigure[$\lambda=0.6$]{
\includegraphics[width=60mm,height=30mm]{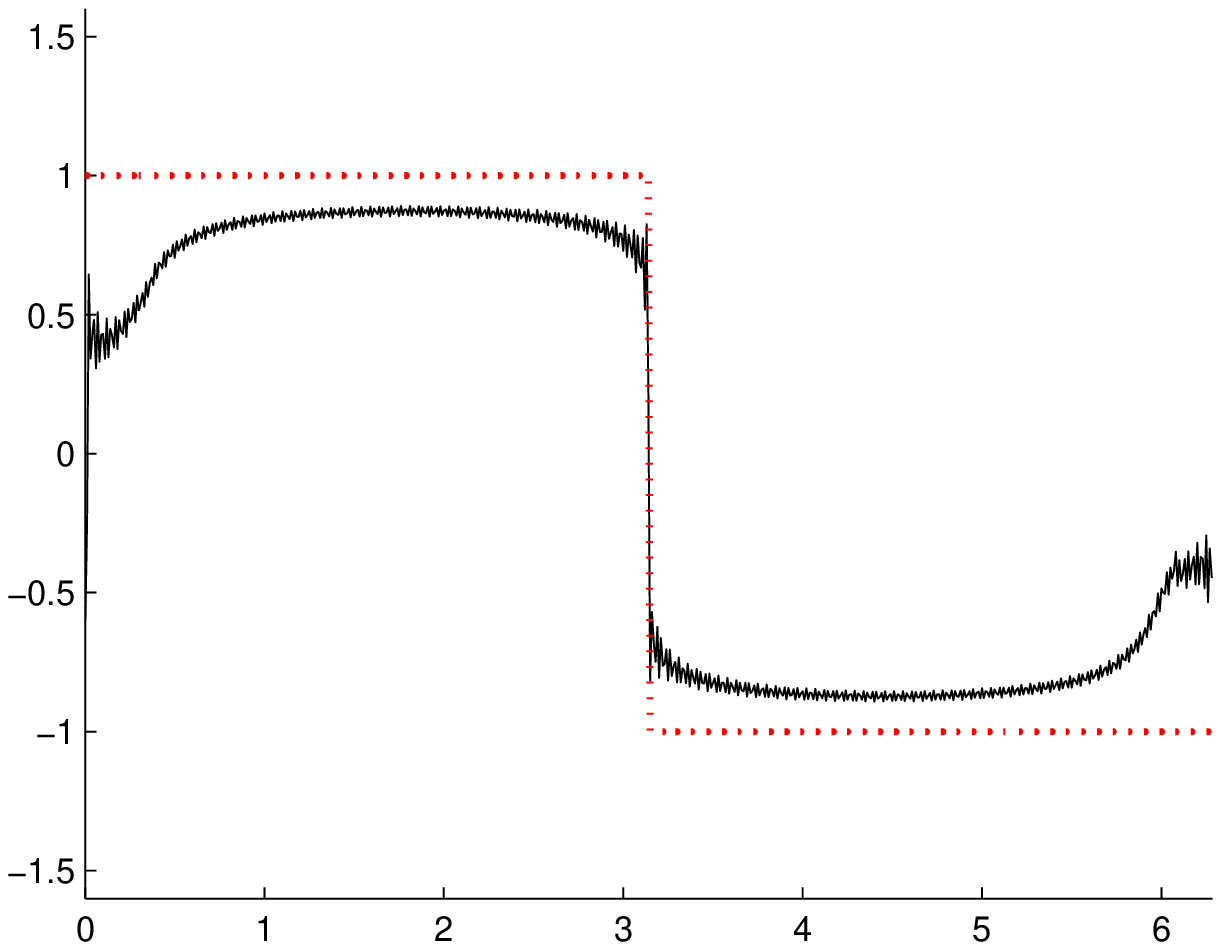}}
\subfigure[$\lambda=0.1$]{
\includegraphics[width=60mm,height=30mm]{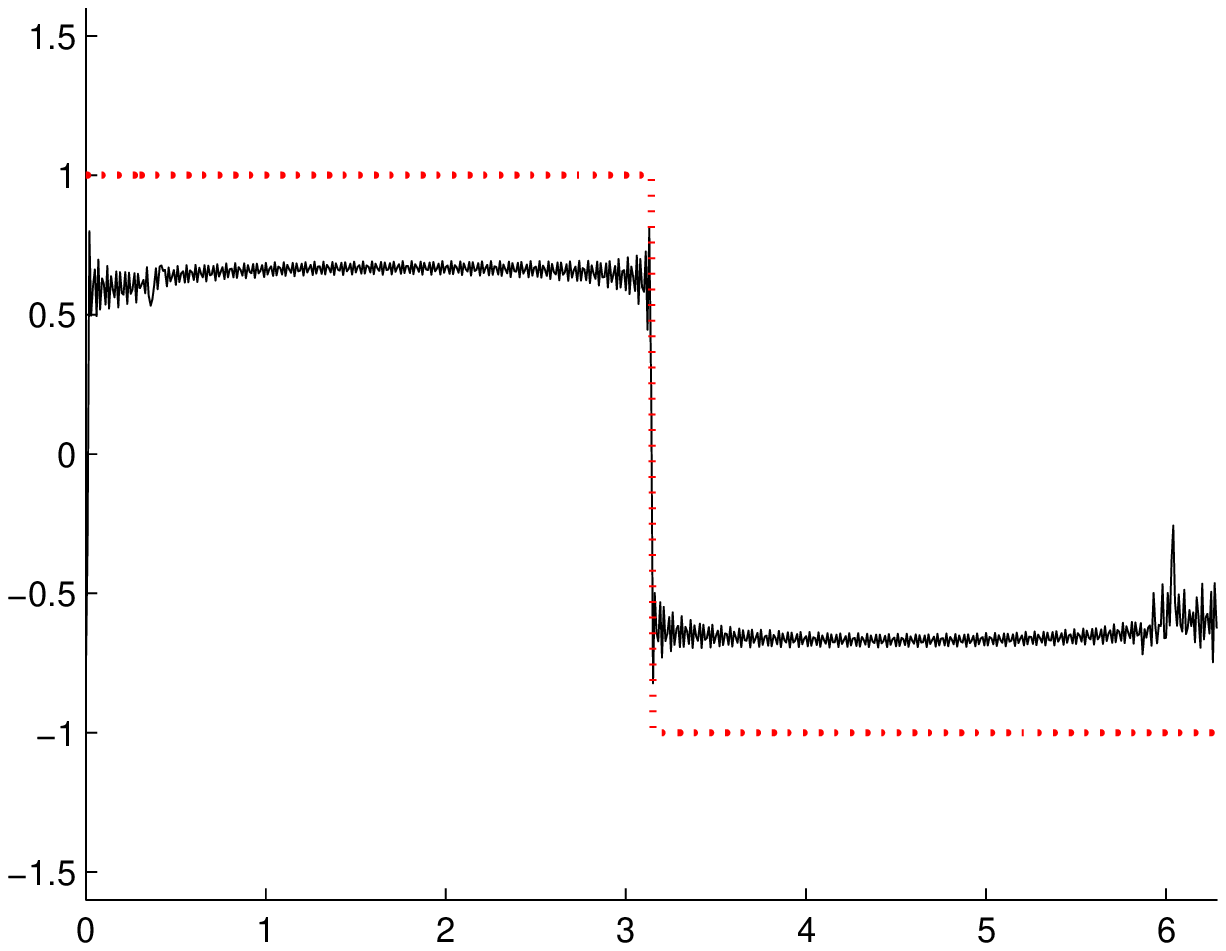}}
\caption{Solutions of system \eqref{3sys} with $N=256$, $T=0.5$, and $\epsilon_N= 0$. The piecewise constant initial datum is
$u_0(x)=\mathrm{sgn}(\pi-x)$.}\label{fff2}
\end{figure}

\begin{remark}
By Proposition \ref{G} there is a positive constant such that
\begin{equation*}
\begin{split}
\int_{\Lambda}\Levy^{\pi_{\lambda}}[u_N(\cdot,t)]\,u_N(\cdot,t)\ \dif x&=-C\,\sum_{|\xi|\leq N}|\xi|^\lambda|\hat{u}_\xi(t)|^2,
\end{split}
\end{equation*}
where right-hand side is a fractional Sobolev semi-norm \cite{Benyu}
\begin{equation*}
\begin{split}
\sum_{|\xi|\leq N}|\xi|^{\lambda}|\hat{u}_\xi(t)|^2=|u_N(\cdot,t)|^2_{H^{\lambda/2}(\Lambda)}.
\end{split}
\end{equation*}
Simple energy estimates can then be used to show that the solutions of \eqref{frac_bur} belong to $H^{\lambda/2}(\Lambda)$, which is more regularity than what
can be expected for general solutions of the pure Burgers' equation ($\mu=0$).
\end{remark}

We now use the SVV method \eqref{2} to work out some approximate solutions of the fractional Burgers' equation \eqref{frac_bur} with $d=1$. Hence $f(u)=u^2/2$
and $\mu=\pi_{\lambda}$ in \eqref{2}. We multiply both sides of \eqref{2} by $e^{-i\xi\,x}$, and integrate over $(0,2\pi)$ to obtain the following system of
ODEs
\begin{equation}\label{3sys}
\begin{split}
\frac{\dif }{\dif t}\hat{u}_\xi (t)+\frac{i\xi }{2}\sum_{\stackrel{|p|,|q|\leq N}{p+q=\xi }}\hat{u}_p(t)\,\hat{u}_q(t) +C_\lambda\,|\xi |^\lambda\hat
u_\xi(t)+\epsilon_N\mathbf{1}_{m_N\leq|\xi|\leq N}|\xi |^2\,\hat Q_{\xi}(t)\,\hat{u}_\xi (t)=0,
\end{split}
\end{equation}
where the Fourier coefficients $\hat Q_{\xi}$ satisfy the assumptions listed in Section \ref{sec:approx} and are chosen as in \cite{Maday/Tadmor} (they vary
continuously between zero and one). In our simulations we have used a fourth order Runge-Kutta solver for \eqref{3sys}.

The results of our numerical simulations can be found in Figure
\ref{fff1} and \ref{fff2}. The results in Figure \ref{fff1}
confirm the convergence of our SVV approximation \eqref{3sys} for all
all values of $\lambda\in(0,2)$. In Figure \ref{fff2} we have solved
the  \eqref{3sys} with $\epsilon_N=0$ (no spectral vanishing
viscosity). For $\lambda>1$,
convergence continues to hold, while for $\lambda< 1$,
convergence fails and spurious Gibbs oscillations appear. This
 is consistent with the theoretical results for fractional conservation laws \cite{Alibaud/Droniou/Vovelle,Droniou/Gallouet/Vovelle}: These
equations admit smooth solutions for $\lambda>1$ (the strong diffusion
case), while shock discontinuities may appear for $\lambda<1$ (the weak
diffusion case).

\begin{proof}[Proof of Proposition \ref{G}]
Let us prove the case $d=1$ first. By Euler's formula, $e^{i\xi
  z}=\cos(\xi z)+i\sin(\xi z)$, we find that
\begin{equation*}
\begin{split}
\int_{|z|<1}\frac{e^{i\xi z}-1-i\xi z}{|z|^{1+\lambda}}\ \dif z&=\int_{|z|<1}\frac{\cos(\xi z)-1}{|z|^{1+\lambda}}\ \dif z +i\int_{|z|<1}\frac{\sin(\xi z)-\xi
z}{|z|^{1+\lambda}}\ \dif z.
\end{split}
\end{equation*}
Taylor expansions show that these integrals are finite. In fact, the
$\sin$-integral is zero since the its integrand is odd.  Integration
by parts then leads to
\begin{equation*}
\begin{split}
\int_{|z|<1}\frac{\cos(\xi z)-1}{|z|^{1+\lambda}}\ \dif z&=2\int_{0}^1\frac{\cos(\xi z)-1}{z^{1+\lambda}}\ \dif z\\
&=-\frac{2}{\lambda z^\lambda}(\cos(\xi z)-1)\Big|_0^1-\frac{2\xi }{\lambda}\int_0^1\frac{\sin(\xi z)}{z^\lambda}\ \dif z\\
&=-\frac{2}{\lambda}(\cos(\xi )-1)-\frac{2\xi }{\lambda}\int_0^1\frac{\sin(\xi z)}{z^\lambda}\ \dif z.
\end{split}
\end{equation*}
Now we consider the integral over $|z|>r$. Again the imaginary part
(the sine part) is zero, and a computation like the one we performed
above reveals that
\begin{equation*}
\begin{split}
\int_{|z|>1}\frac{e^{i\xi z}-1}{|z|^{1+\lambda}}\ \dif z =\frac{2}{\lambda}(\cos(\xi )-1)-\frac{2\xi }{\lambda}\int_1^\infty\frac{\sin(\xi z)}{z^\lambda}\ \dif
z.
\end{split}
\end{equation*}
Note the + sign of the cosine-term! We add these two equations and find that
\begin{equation*}
\begin{split}
G^{\pi_{\lambda}}(\xi )=-\frac{2\xi c_\lambda}{\lambda}\int_0^\infty\frac{\sin(\xi z)}{z^\lambda}\ \dif z.
\end{split}
\end{equation*}
The integral $\int_0^\infty z^{-\lambda}\sin(\xi z)\, \dif z$ is finite and positive for all $\lambda\in(0,2)$ (cf.~\cite{Courant} for details). Whenever $\xi
>0$, we can use the change of variable $\xi z\rightarrow x$ to deduce that
\begin{equation*}
\begin{split}
\int_0^\infty\frac{\sin(\xi z)}{z^\lambda}\ \dif z=\xi ^{\lambda-1}\int_0^\infty\frac{\sin x}{x^\lambda}\ \dif x
\end{split}
\end{equation*}
and thus
\begin{equation*}
\begin{split}
G^{\pi_{\lambda}}(\xi )&=-\frac{2c_\lambda}{\lambda}\,\xi ^{\lambda}\int_0^\infty\frac{\sin x}{x^\lambda}\ \dif x.
\end{split}
\end{equation*}
When $\xi <0$, we use the relation $\sin(-\xi x)=-\sin(\xi x)$ to obtain
\begin{equation*}
\begin{split}
G^{\pi_{\lambda}}(\xi )&=-\frac{2c_\lambda}{\lambda}\,|\xi |^{\lambda}\int_0^\infty\frac{\sin x}{x^\lambda}\ \dif x,
\end{split}
\end{equation*}
and the conclusion for $d=1$ follows.

When $d>1$ we use polar coordinates $x=ry$ for $r>0$ and $|y|=1$, and
we find that
\begin{equation*}
\begin{split}
\int_{|z|<1}\frac{e^{i\xi\cdot z}-1-i\xi\cdot z}{|z|^{d+\lambda}}\
\dif z&=\int_{|y|=1}\int_{0}^1\frac{\cos(\xi\cdot y\,
  r)-1}{r^{d+\lambda}}\ r^{d-1} \dif
r\,\dif S_y,\\
\int_{|z|>1}\frac{e^{i\xi\cdot z}-1}{|z|^{d+\lambda}}\ \dif z&=\int_{|y|=1}\int_{1}^\infty\frac{\cos(\xi\cdot y\, r)-1}{r^{1+\lambda}}\ \dif
r\,\dif S_y.
\end{split}
\end{equation*}
Proceeding as in the $d=1$ case for the $r$-integral with $y$ fixed,
we find that
\begin{equation*}
\begin{split}
G^{\pi_{\lambda}}(\xi )&=-\frac{2c_{\lambda}}{\lambda}\int_{|y|=1}|\xi\cdot y|^{\lambda}\ \dif S_y\int_0^\infty\frac{\sin x}{x^\lambda}\ \dif x\\
&=-\frac{2c_{\lambda}}{\lambda}\,|\xi|^\lambda\int_{|y|=1}\left|\frac{\xi}{|\xi|}\cdot y\right|^{\lambda}\ \dif S_y\int_0^\infty\frac{\sin x}{x^\lambda}\ \dif
x.
\end{split}
\end{equation*}
By symmetry, the value of the $y$-integral is
the same for any $\xi$. Therefore,
\begin{equation*}
\begin{split}
\int_{|y|=1}\left|\frac{\xi}{|\xi|}\cdot y\right|^{\lambda}\ \dif S_y=\int_{|y|=1}\left|y\cdot y\right|^{\lambda}\ \dif S_y=\int_{|y|=1}\dif S_y.
\end{split}
\end{equation*}
The proof for the case $d>1$ is now complete.
\end{proof}

\section{Extension to asymmetric measures $\mu$}\label{sec:asymm_approx}
In this section we show how to modify the arguments of the previous sections to obtain results for a large class of non-symmetric measures $\mu$ including all
the L\'{e}vy measures used in finance. A careful look at the previous arguments shows that symmetry of $\mu$ is used for the sole purpose of having a sign of
the fractional term in the energy inequality (see \eqref{pp3}) in order to prove Theorems \ref{Th:2.1} and \ref{thm2}. This fractional term is
 \begin{equation}\label{lll}
\begin{split}
\iint_{D_T}\Levy^{\mu}[u_N]\,\partial_x^{2\alpha}u_N\ \dif x\,\dif t=\sum_{|\xi|\leq N}G^\mu(\xi)\,|\xi^\alp|^{2}\int_0^T|\hat u_\xi(t)|^2\ \dif t,
\end{split}
\end{equation}
and it is non-positive when $\mu$ is symmetric. In
the general case the sign of the fractional term \eqref{lll} is
unknown, but everything still works if we assume that
$$\mu=\mu_s+\mu_n,$$
for $\mu_s,\mu_n$ satisfying \eqref{rev_2} (i.e. we assume
\eqref{splitt} and \eqref{rev_2}). Note that in this case, we may split
the weights in \eqref{weights} into their symmetric and non-symmetric
parts,
$$G^{\mu}(\xi)=G^{\mu_s}(\xi)+G^{\mu_n}(\xi),$$
 where $G^{\mu_s}(\xi)$ is again real and non-positive, and by \eqref{rev_2},
\begin{equation}\label{rev_1}
\begin{split}
|G^{\mu_n}(\xi)|=\bigg|\int_{|z|>0}e^{i\xi\cdot z}-1-i\xi\cdot z\,\mathbf{1}_{|z|<1}\ \dif\mu_n(z)\bigg|\leq C_{n}\Big(1+|\xi|\Big).
\end{split}
\end{equation}
The main result of this section is the following:

\begin{theorem}\emph{(Convergence with rate)}\label{rate2}
Let  $(\mathbf{A}.1)$--$(\mathbf{A}.9)$, \eqref{splitt} and \eqref{rev_2} hold, $u_N$ be the solution of the SVV method \eqref{2}, and $u$ be an entropy
solution of \eqref{1}. Then,
\begin{equation*}
\begin{split}
\|u(\cdot,T)-u_N(\cdot,T)\|_{L^1(\Lambda)}\leq C\,\sqrt{\epsilon_N}.
\end{split}
\end{equation*}
\end{theorem}

To prove this result, we have to modify the arguments of the previous
sections. In view of the above discussion the key result to obtain is
a version of Theorem \ref{Th:2.1} for measures $\mu$
satisfying \eqref{splitt} and \eqref{rev_2}:

\begin{theorem}\label{Th:8.1}
Assume (\textbf{A}.1)--(\textbf{A}.7), \eqref{splitt}, \eqref{rev_2}
hold, and let $u_N$ be the solution of the SVV approximation
\eqref{2}. Then there exists a constant
$\tilde{\mathcal{B}}_s$ (proportional to $1+\Pi_{k=1}^s\mathcal{K}_s$ for $s\geq 1$ and to $\|u_N\|_{L^\infty}$ for
$s=0$, see Theorem \ref{Th:2.1}) such that
\begin{equation*}
\begin{split}
&\epsilon^s_N\|\partial_x^su_N(\cdot,t)\|_{L^2(\Lambda)}+\epsilon_N^{s+\frac{1}{2}}\|\partial_x^{s+1}u_N\|_{L^2(D_T)}
\leq \tilde{\mathcal{B}}_s+ 4\epsilon_N^s\|\partial_x^su_N(\cdot,0)\|_{L^2(\Lambda)}.
\end{split}
\end{equation*}
\end{theorem}

We prove this result at the end of this section. Now if we also assume that (\textbf{A}.8) and (\textbf{A}.9) hold, then it easily follows that Theorem
\ref{thm2} still holds if we replace $\mathcal B_s$ by $\tilde{\mathcal B}_s$. At this point the reader may easily check that {\em all} the other results also
hold if we everywhere replace $\mathcal B_s$ by $\tilde{\mathcal B}_s$ -- and hence Theorem \ref{rate2} follows.

\begin{remark}
\label{rem8} A L\'{e}vy measure $\mu$ defined by
$$\dif\mu=g(z)\,\dif\pi_\lambda(z),$$
(see \eqref{asym}) can be written as $\mu=\mu_s+\mu_n$ where
$$\dif\mu_s=g(z)\wedge g(-z)\,\dif\pi_\lambda\quad\text{and}\quad
\dif\mu_n=[g(z)-g(z)\wedge g(-z)]\,\dif\pi_\lambda.$$ Note that $\mu_s,\mu_n\geq0$, $\mu_s$ is symmetric, and that $\mu_n$ satisfies the integrability
condition in \eqref{rev_2} if $g$ is locally Lipschitz: Let $g_n(z)=g(z)-g(z)\wedge g(-z)$ and note that $g_n(0)=0$, hence $g_n(z)=|g_n(z)-g_n(0)|\leq C\,|z|$
for $|z|<1$.
\end{remark}

We now show how to modify the proof of Theorem
\ref{Th:2.1} to prove Theorem \ref{Th:8.1}.
\begin{proof}[Proof of Theorem \ref{Th:8.1}]
Once again we use the shorthand $\|\cdot\|$ instead of $\|\cdot\|_{L^2(\Lambda)}$, and rewrite the SVV approximation \eqref{2} as in \eqref{f1} and
\eqref{T_7.1}. Note that \eqref{h1}
and \eqref{h2} holds for general measures $\mu$, so we find that
$$\int_\Lambda \Levy^{\mu}[u_N]\,u_N\ \dif x\leq \int_\Lambda \Levy^{\mu}[u^2_N]\ \dif x = 0.$$
Hence, spatial integration of \eqref{T_7.1} against $u_N$ yields
\begin{equation*}
\begin{split}
&\frac{1}{2}\frac{\dif}{\dif t}\|u_N\|^2+\epsilon_N\,\|\partial_xu_N\|^2\\
&\leq\epsilon_N\|u_N\|\bigg\|\sum_{j,k=1}^d\del_j\del_kR_N^{j,k}\ast u_N\bigg\|+\sum_{j=1}^d\|\partial_ju_N\|\|(I-P_N)f_j(u_N)\|,
\end{split}
\end{equation*}
and the conclusion in the case $s=0$ follows exactly as in the first part of
the proof of Theorem \ref{Th:2.1}.

Now let $s>0$, and note that by \eqref{rev_1} and Young's inequality,
\begin{equation*}
\begin{split}
\int_\Lambda  \partial_x^{2\alpha}u_N\ \Levy^{\mu_n}[u_N]\ \dif
x&=\sum_{|\xi|\leq N}(-i\xi)^{2\alp}\,G^{\mu_n}(\xi)\,|\hat
u_\xi(t)|^2\\
&\leq \sum_{|\xi|\leq N}C_n\Big(1+|\xi|\Big)|\xi^\alpha|^2\,|\hat u_\xi(t)|^2\\
&\leq \sum_{|\xi|\leq
  N}\left(C_n+\frac{\epsilon_N}{4}\,|\xi|^2+\frac{C_n^2}{\epsilon_N}\right)|\xi^\alpha|^2|\hat
u_\xi(t)|^2.
\end{split}
\end{equation*}
If we take this into account and perform spatial integration of
\eqref{f1} against
$\partial_x^{2\alp}u_N$ for some multi-index $\alp$, we find the
following modified version of \eqref{f2},

\begin{equation*}
\begin{split}&\frac{1}{2}\frac{\dif}{\dif t}\|\partial_x^\alp u_N\|^2
-\sum_{|\xi|\leq N}G^{\mu_s}(\xi)|\xi^\alp|^{2}|\hat u_\xi(t)|^2+\frac{3\,\epsilon_N}{4}\|\del^\alp_x\partial_xu_N\|^2\\
&\leq\epsilon_N\|\partial_x^{\alp}u_N\|\bigg\|\sum_{j,k=1}^d\del_j\del_kR_N^{j,k}\ast
\partial_x^{\alp}u_N\bigg\|\\
&\quad+\|\partial_x^{\alp}\del_xu_N\|\|\partial_x^{|\alp|-1}\partial_x
\cdot P_Nf(u_N)\|+\frac{2\,C_n^2}{\epsilon_N}\|\partial_x^{\alp}u_N\|^2.
\end{split}
\end{equation*}
As in the proof of Theorem \ref{Th:2.1}, we now use \eqref{hh2} and
Young's inequality to bound the first and second
term on the right hand side. The result is that
\begin{equation*}
\begin{split}
&\frac12\frac{\dif}{\dif t}\|\partial_x^\alp u_N\|^2-\sum_{|\xi|\leq N}G^{\mu_s}(\xi)|\xi^\alp|^{2}|\hat u_\xi(t)|^2+\frac{\epsilon_N}{2}\|\del^\alp_x\partial_xu_N\|^2\\
&\qquad\qquad\qquad\qquad\qquad\leq C\,\|\partial_x^{\alp}u_N\|^2+\frac{1}{\epsilon_N}\|\partial_x^{|\alp|}
P_Nf(u_N)\|^2+\frac{2\,C_n^2}{\epsilon_N}\|\partial_x^{\alp}u_N\|^2.
\end{split}
\end{equation*}
Now we sum over all $|\alp|= s$ to find that
\begin{equation*}
\begin{split}
&\frac12\frac{\dif}{\dif t}\|\partial_x^s u_N\|^2-\sum_{|\alp|= s}\sum_{|\xi|\leq N}G^{\mu_s}(\xi)|\xi^\alp|^{2}|\hat u_\xi(t)|^2
+\frac{\epsilon_N}{2}\|\del^{s+1}_xu_N\|^2\\
&\qquad\qquad\qquad\qquad\qquad\leq C\,\|\partial_x^{s}u_N\|^2+\frac{d^s}{\epsilon_N}\|\partial_x^{s}
P_Nf(u_N)\|^2+\frac{2\,C_n^2}{\epsilon_N}\|\partial_x^{s}u_N\|^2.
\end{split}
\end{equation*}
Thanks to \eqref{T_2.2} and \eqref{T_2.3},
\begin{equation*}
\begin{split}
\|\partial_x^{s} P_Nf(u_N)\|\leq \mathcal{K}_{s}\,\|\partial_x^{s}u_N\|+\frac{\mathcal{K}_{s+1}}{N}\,\|\partial_x^{s+1}u_N\|,
\end{split}
\end{equation*}
and hence
\begin{equation*}
\begin{split}
\frac12\frac{\dif}{\dif t}\|\partial_x^s u_N\|^2-\sum_{|\alp|= s}\sum_{|\xi|\leq N}G^{\mu_s}(\xi)|\xi^\alp|^{2}|\hat u_\xi(t)|^2+\left(\frac{\epsilon_N}{2}
-\frac{2d^s\mathcal{K}^2_{s+1}}{N^2\epsilon_N}\right)\|\del_x^{s+1}u_N\|^2\\
\leq
\left(C+\frac{2\,C_n^2+2d^s\mathcal{K}_s^2}{\epsilon_N}\right)\,\|\partial_x^{s}u_N\|^2\leq\frac{2\,C_n^2+3d^s\mathcal{K}_s^2}{\epsilon_N}\,\|\partial_x^{s}u_N\|^2,
\end{split}
\end{equation*}
where the last inequality holds for $N$ big enough.

To conclude, we use \eqref{ass} to obtain
\begin{equation*}
\begin{split}
\frac12\|\partial_x^su_N(\cdot,t)\|^2-\sum_{|\alp|= s}\sum_{|\xi|\leq N}G^{\mu_s}(\xi)|\xi^\alp|^{2}\int_0^t|\hat u_\xi(\tau)|^2\, \dif \tau
+\frac{\epsilon_N}{4}
\|\partial_x^{s+1}u_N\|^2_{L^2(D_T)}\\
\leq \frac{2\,C_n^2+3d^s\mathcal{K}_s^2}{\epsilon_N}\,\|\partial_x^{s}u_N\|^2_{L^2(D_T)}+\frac12\|\partial_x^su_N(\cdot,0)\|^2.
\end{split}
\end{equation*}
The proof is now complete since by induction on $s$,
\begin{equation*}
\begin{split}
\|\partial_x^{s}u_N\|^2_{L^2(D_T)}\leq C\,\tilde{\mathcal{B}}^2_{s-1} \epsilon_N^{-(2s-1)}.
\end{split}
\end{equation*}
\end{proof}

\appendix

\section{Proof of Theorem \ref{th:uniqueness}}\label{app:1}
Let us take $\varphi=\psi(x,y,t,s)$, $u=u(x,t)$ and $v=v(y,s)$. We set
$k=v(y,s)$ in the entropy inequality for $u(x,t)$, and integrate over all
$(y,s)\in Q_T$ to obtain
\begin{equation*}
\begin{split}
&\iint_{D_T}\iint_{D_T}\eta(u(x,t),v(y,s))\,\partial_t\psi(x,y,t,s)\\
&\qquad\qquad+q(u(x,t),v(y,s))\cdot\partial_x\psi(x,y,t,s)\\
&\qquad\qquad+\eta(u(x,t),v(y,s))\,\Levy_r^{\ast,\mu}[\psi(\cdot,y,t,s)](x)\\
&\qquad\qquad+\eta'(u(x,t),v(y,s))\,\Levy^{\mu,r}[u(\cdot,t)](x)\,\psi(x,y,t,s)\\
&\qquad\qquad+\eta(u(x,t),v(y,s))\,\gamma_\mu^r\cdot\partial_x\psi(x,y,t,s)\ \dif x\,\dif t\,\dif y\,\dif s\geq0.
\end{split}
\end{equation*}
In the entropy inequality for $v(y,s)$, we set $k=u(x,t)$ and integrate with
respect to $(x,t)$ to find that
\begin{equation*}
\begin{split}
&\iint_{D_T}\iint_{D_T}\eta(u(x,t),v(y,s))\,\partial_s\psi(x,y,t,s)\\
&\qquad\qquad+q(u(x,t),v(y,s))\cdot\partial_y\psi(x,y,t,s)\\
&\qquad\qquad+\eta(u(x,t),v(y,s))\,\Levy_r^{\ast,\mu}[\psi(x,\cdot,t,s)](y)\\
&\qquad\qquad-\eta'(u(x,t),v(y,s))\,\Levy^{\mu,r}[v(\cdot,s)](y)\,\psi(x,y,t,s)\\
&\qquad\qquad+\eta(u(x,t),v(y,s))\,\gamma_\mu^r\cdot\partial_y\psi(x,y,t,s)\ \dif y\,\dif s\,\dif x\,\dif t\geq0.
\end{split}
\end{equation*}
In the following we need the $\R^{2d}$-operators
\begin{equation*}
\begin{split}
\tilde \Levy^{\mu,r}[\phi(\cdot,\cdot)](x,y)&=\int_{|z|>r}\phi(x+z,y+z)-\phi(x,y)\ \dif \mu(z),\\
\tilde \Levy^{\ast,\mu,r}[\phi(\cdot,\cdot)](x,y)&=\int_{|z|>r}\phi(x-z,y-z)-\phi(x,y)\ \dif \mu(z).
\end{split}
\end{equation*}
With these definitions in mind, we add the two inequalities above and change the order of integration to find that
\begin{equation*}
\begin{split}
&\iint_{D_T}\iint_{D_T}\eta(u(x,t),v(y,s))\, (\partial_t+\partial_s)\psi(x,y,t,s)\\
&\qquad\qquad+q(u(x,t),v(y,s))\cdot(\partial_x+\partial_y)\psi(x,y,t,s)\\
&\qquad\qquad+\eta(u(x,t),v(y,s))\,\Levy_r^{\ast,\mu}[\psi(\cdot,y,t,s)](x)\\
&\qquad\qquad+\eta(u(x,t),v(y,s))\, \Levy_r^{\ast,\mu}[\psi(x,\cdot,t,s)](y)\\
&\qquad\qquad+\eta'(u(x,t),v(y,s))\,\tilde \Levy^{\mu,r}[u(\cdot,t)-v(\cdot,s)](x,y)\,\psi(x,y,t,s)\\
&\qquad\qquad+\eta(u(x,t),v(y,s))\,\gamma_\mu^r\cdot(\partial_x+\partial_y)\psi(x,y,t,s)\ \dif w\geq 0.
\end{split}
\end{equation*}
Here and in the following we use the shorthand $\dif w=\dif x\,\dif
t\,\dif y\,\dif s$. Note that
\begin{equation*}
\begin{split}
\eta'(u(x,t),v(y,s))\, \tilde \Levy^{\mu,r}[u(\cdot,t)-v(\cdot,s)](x,y)\leq \tilde \Levy^{\mu,r}[\eta(u(\cdot,t),v(\cdot,s))](x,y).
\end{split}
\end{equation*}
Moreover, using the change of variables $(x,y)\rightarrow(x-z,y-z)$,
\begin{equation*}
\begin{split}
&\iint_{D_T}\iint_{D_T}\psi(x,y,t,s)\, \tilde \Levy^{\mu,r}[\eta(u(\cdot,t),v(\cdot,s))](x,y)\ \dif w\\
&=\int_{|z|>r}\int_{0}^T\int_{z+\Lambda}\int_{0}^T\int_{z+\Lambda}\eta(u(x,t),v(y,s))\,\psi(x-z,y-z,t,s)\ \dif w\,\dif\mu(z)\\
&\quad-\int_{|z|>r}\iint_{D_T}\iint_{D_T}\eta(u(x,t),v(y,s))\,\psi(x,y,t,s)\ \dif w\,\dif\mu(z),
\end{split}
\end{equation*}
which by periodicity and the definition of $\tilde \Levy^{\ast,\mu,r}$ equals to
\begin{equation*}
\begin{split}
&\int_{|z|>r}\int_{0}^T\int_{\Lambda}\int_{0}^T\int_{\Lambda}\eta(u(x,t),v(y,s))\,\psi(x-z,y-z,t,s)\ \dif w\,\dif\mu(z)\\
&\quad-\int_{|z|>r}\iint_{D_T}\iint_{D_T}\eta(u(x,t),v(y,s))\,\psi(x,y,t,s)\ \dif w\,\dif\mu(z)\\
&=\iint_{D_T}\iint_{D_T}\eta(u(x,t),v(y,s))\, \tilde \Levy^{\ast,\mu,r}[\psi(\cdot,\cdot,t,s)](x,y)\ \dif w.
\end{split}
\end{equation*}
Therefore we have proved so far that
\begin{equation*}
\begin{split}
&\iint_{D_T}\iint_{D_T}\eta(u(x,t),v(y,s))\, (\partial_t+\partial_s)\psi(x,y,t,s)\\
&\qquad\qquad+q(u(x,t),v(y,s))\cdot(\partial_x+\partial_y)\psi(x,y,t,s)\\
&\qquad\qquad+\eta(u(x,t),v(y,s))\,\Levy_r^{\ast,\mu}[\psi(\cdot,y,t,s)](x)\\
&\qquad\qquad+\eta(u(x,t),v(y,s))\, \Levy_r^{\ast,\mu}[\psi(x,\cdot,t,s)](y)\\
&\qquad\qquad+\eta(u(x,t),v(y,s))\,\tilde \Levy^{\ast,\mu,r}[\psi(\cdot,\cdot,t,s)](x,y)\\
&\qquad\qquad+\eta(u(x,t),v(y,s))\,\gamma_\mu^r\cdot(\partial_x+\partial_y)\psi(x,y,t,s)\ \dif w\geq 0.
\end{split}
\end{equation*}
We now send $r\rightarrow0$, remembering the definition of $\gamma_\mu^r$
and defining
\begin{equation*}
\begin{split}
&\tilde \Levy^{\ast,\mu}[\phi(\cdot,\cdot)](x,y)\\
&=\int_{|z|>0}\phi(x-z,y-z)-\phi(x,y)+z\cdot(\partial_x+\partial_y)\phi(x,y)\,\mathbf{1}_{|z|<1}\ \dif \mu(z).
\end{split}
\end{equation*}
The result is
\begin{equation}\label{r2}
\begin{split}
&\iint_{D_T}\iint_{D_T}\eta(u(x,t),v(y,s))\, (\partial_t+\partial_s)\psi(x,y,t,s)\\
&\qquad\qquad+q(u(x,t),v(y,s))\cdot(\partial_x+\partial_y)\psi(x,y,t,s)\\
&\qquad\qquad+\eta(u(x,t),v(y,s))\,\tilde \Levy^{\ast,\mu}[\psi(\cdot,\cdot,t,s)](x,y)\ \dif w\geq 0.
\end{split}
\end{equation}

To conclude, we show how to derive the $L^1$-contraction
\eqref{contraction_periodic} from this inequality by choosing the test
function $\psi$ as
\begin{equation}\label{test_func}
\begin{split}
\psi(x,y,t,s)=\hat\omega_\rho\left(\frac{x-y}{2}\right)\omega_{\delta}\left(\frac{t-s}{2}\right)\phi(t),\quad\rho,\delta>0,
\end{split}
\end{equation}
where $\omega_\delta(\tau)=\frac{1}{\delta}\,\omega(\frac{\tau}{\delta})$ for a nonnegative $\omega\in C_c^\infty(\mathbb{R})$ satisfying
\begin{equation*}
\begin{split}
\omega(-\tau)=\omega(\tau),\quad \omega(\tau)=0 \text{ for all $|\tau|\geq 1$,$\quad$and}\quad \int_{\mathbb{R}}\omega(\tau)\, \dif \tau=1,
\end{split}
\end{equation*}
while $\hat\omega_\rho(x)=\bar{\omega}_\rho(x_1)\cdots \bar{\omega}_\rho(x_d)$ with $\bar\omega_{\rho}(\cdot)$ such that
\begin{equation*}
\begin{split}
\bar\omega_{\rho}(\tau)=\sum_{k\in\mathbb{Z}}\omega_\rho(\tau+2\pi k).
\end{split}
\end{equation*}
Note that $\hat\omega_{\rho}$ is periodic in each
coordinate direction. By a direct computation,
\begin{equation*}
\begin{split}
(\partial_t+\partial_s)\psi(x,y,t,s)&=\hat\omega_\rho\left(\frac{x-y}{2}\right)\omega_{\delta}\left(\frac{t-s}{2}\right)\,\phi'(t),\\
(\partial_x+\partial_y)\psi(x,y,t,s)&=0,\\
\tilde \Levy^\ast[\psi(\cdot,\cdot,t,s)](x,y)&=0.
\end{split}
\end{equation*}
Thus, with this test function $\psi$ at hand, inequality \eqref{r2} becomes
\begin{equation}\label{j2}
\begin{split}
&\iint_{D_T}|u(x,t)-v(y,s)|\,\hat\omega_\rho\left(\frac{x-y}{2}\right)\omega_{\delta}\left(\frac{t-s}{2}\right)\,\phi'(t)\ \dif w\geq0.
\end{split}
\end{equation}
We then go to the limit as $(\rho,\delta)\rightarrow 0$ to find that
\begin{equation}\label{giulia3}
\begin{split}
\iint_{D_T}|u(x,t)-v(x,t)|\, \phi'(t)\ \dif x\,\dif t\geq0.
\end{split}
\end{equation}
To conclude the proof we now take $\phi=\chi_\mu$ for
\begin{align}\label{chi_fun}
\chi_\mu(t)&=\int_{-\infty}^{t}(\omega_\mu(\tau-t_1)-\omega_\mu(\tau-t_2))\ \dif\tau,\qquad 0<t_1<t_2<T.
\end{align}
Loosely speaking, the function $\chi_\mu$ is a smooth approximation of the indicator function $\mathbf{1}_{(t_1,t_2)}$ which is zero near $t=0$ and $t=T$ for
$\mu>0$ small. Since $$\chi_\mu'(t)=\omega_\mu(t-t_1)-\omega_\mu(t-t_2),$$ inequality \eqref{giulia3} reduces to
\begin{equation*}
\iint_{Q_T}|u(x,t)-v(x,t)|\,\omega_\mu(t-t_2)\ \dif t\,\dif x\leq \iint_{Q_T}|u(x,t)-v(x,t)|\,\omega_\mu(t-t_1)\ \dif t\,\dif x.
\end{equation*}
By the integrability of $u$ and $v$ and Fubini's theorem, the function
\begin{equation*}
\Phi(t)=\int_{\Lambda}|u(x,t)-v(x,t)|\ \dif x\in L^1(0,T),
\end{equation*}
and we may write the above inequality as a convolution
$$\Phi*\omega_\mu(t_2) \leq \Phi*\omega_\mu(t_1).$$
By standard
properties of convolutions, $\Phi*\omega_\mu(t)\ra \Phi(t)$ a.e. $t$
as $\mu\ra0$. Hence,
\begin{equation*}
\|(u-v)(\cdot,t_2)\|_{L^1(\Lambda)}\leq\|(u-v)(\cdot,t_1)\|_{L^1(\Lambda)} \quad\text{for a.e.~$t_1,t_2\in(0,T)$}.
\end{equation*}
Finally, the theorem follows from renaming $t_2$ and using part \emph{iii)} in Definition \ref{def:alibaud_periodic} to send $t_1\rightarrow0$.

\section{Proof of Theorem \ref{lemma_rate}}\label{app:2}
\label{App:VV}
The vanishing viscosity problem \eqref{prob:viscous} has a unique
classical solution $u_\epsilon$ for $\eps>0$, see Remark \ref{remVV}.
If we multiply \eqref{prob:viscous} by $\eta'(u_\epsilon)$ for
any smooth convex function $\eta$, use standard manipulations on the
conservation law part combined with the inequalities
\begin{gather*}
\eta'(u_{\epsilon})\Levy^\mu[u_{\epsilon}]=\eta'(u_{\epsilon})\bigg(\Levy^\mu_r[u_{\epsilon}]+\Levy^{\mu,r}[u_{\epsilon}]\bigg)
\leq\Levy_r^\mu[\eta(u_{\epsilon})]+\eta'(u_{\epsilon})\Levy^{\mu,r}[u_{\epsilon}],\\
\eta'(u_{\epsilon})\,\Delta u_{\epsilon}=\Delta\eta(u_{\epsilon})- \epsilon\,\eta''(u_{\epsilon})|\partial_{x}
u_{\epsilon}|^2\leq\Delta\eta(u_{\epsilon}),
\end{gather*}
we find, after integration against any nonnegative test function $\phi$, that $u_\epsilon$ satisfies the (entropy) inequality
\begin{equation*}
\begin{split}
&\iint_{D_T}\eta(u_{\epsilon},k)\,\partial_t\varphi+q(u_{\epsilon},k)\cdot\partial_x\varphi+\eta(u_{\epsilon},k)\,\Levy_r^{\ast,\mu}[\varphi]
+\eta'(u_{\epsilon},k)\,\Levy^{\mu,r}[u_{\epsilon}]\\
&\qquad\qquad\qquad\qquad\qquad\qquad+\eta(u_{\epsilon},k)\,\gamma_\mu^r\cdot\partial_x\varphi+\epsilon\,\eta(u_\epsilon,k)\,\Delta \varphi\ \dif x\, \dif
t\geq0.
\end{split}
\end{equation*}

 From this inequality we proceed as in the proof of the
$L^1$-contraction (Theorem \ref{th:uniqueness}). We take $u=u(x,t)$,
$u_\epsilon=u_\epsilon(y,s)$, and find the inequalities
\begin{equation*}
\begin{split}
&\iint_{D_T}\iint_{D_T}\eta(u(x,t),u_{\epsilon}(y,s))\,\partial_t\psi(x,y,t,s)\\
&\qquad\qquad+q(u(x,t),u_{\epsilon}(y,s))\cdot\partial_x\psi(x,y,t,s)\\
&\qquad\qquad+\eta(u(x,t),u_{\epsilon}(y,s))\,\Levy_r^{\ast,\mu}[\psi(\cdot,y,t,s)](x)\\
&\qquad\qquad+\eta'(u(x,t),u_{\epsilon}(y,s))\,\Levy^{\mu,r}[u(\cdot,t)](x)\,\psi(x,y,t,s)\\
&\qquad\qquad+\eta(u(x,t),u_{\epsilon}(y,s))\,\gamma_\mu^r\cdot\partial_x\psi(x,y,t,s)\ \dif w\geq0.
\end{split}
\end{equation*}
and
\begin{equation*}
\begin{split}
&\iint_{D_T}\iint_{D_T}\eta(u(x,t),u_{\epsilon}(y,s))\,\partial_s\psi(x,y,t,s)\\
&\qquad\qquad+q(u(x,t),u_{\epsilon}(y,s))\cdot\partial_y\psi(x,y,t,s)\\
&\qquad\qquad+\eta(u(x,t),u_{\epsilon}(y,s))\,\Levy_r^{\ast,\mu}[\psi(x,\cdot,t,s)](y)\\
&\qquad\qquad-\eta'(u(x,t),u_{\epsilon}(y,s))\,\Levy^{\mu,r}[u_{\epsilon}(\cdot,s)](y)\,\psi(x,y,t,s)\\
&\qquad\qquad+\eta(u(x,t),u_{\epsilon}(y,s))\,\gamma_\mu^r\cdot\partial_y\psi(x,y,t,s)\
\dif w\\
&\qquad\qquad+ \epsilon\,\eta(u(x,t),u_\epsilon(y,s))\,\Delta_y \psi(x,y,t,s)\ \dif w\geq0.
\end{split}
\end{equation*}
As in the proof of Theorem \ref{th:uniqueness}, we add and manipulate
these to get (see \eqref{r2})
\begin{equation*}
\begin{split}
&\iint_{D_T}\iint_{D_T}\eta(u(x,t),u_{\epsilon}(y,s))\, (\partial_t+\partial_s)\psi(x,y,t,s)\\
&\qquad\qquad+q(u(x,t),u_{\epsilon}(y,s))\cdot(\partial_x+\partial_y)\psi(x,y,t,s)\\
&\qquad\qquad+\eta(u(x,t),u_{\epsilon}(y,s))\,\tilde{\Levy}^{\ast,\mu}[\psi(\cdot,\cdot,t,s)](x,y)\\
&\qquad\qquad+ \epsilon\,\eta(u(x,t),u_\epsilon(y,s))\,\Delta_y \psi(x,y,t,s)\ \dif w\geq0.
\end{split}
\end{equation*}
We now take the test function $\psi$ as in \eqref{test_func} and find that (see \eqref{j2})
\begin{equation}\label{dbl}
\begin{split}
&-\iint_{D_T}\iint_{D_T}|u(x,t)-v(y,s)|\,\hat\omega_\rho\left(\frac{x-y}{2}\right)\omega_{\delta}\left(\frac{t-s}{2}\right)\,\phi'(t)\ \dif w\\
&\qquad\qquad\qquad\qquad\leq \epsilon\iint_{D_T}\iint_{D_T}\eta(u(x,t),u_\epsilon(y,s))\,\Delta_y \psi(x,y,t,s)\ \dif w.
\end{split}
\end{equation}
After an integration by parts, the right-hand side (R.H.S.) is bounded by
\begin{equation*}
\begin{split}
\mathrm{R.H.S.} & \leq \epsilon\iint_{D_T}\iint_{D_T}\Big|\partial_{y}|u(x,t)-u_\epsilon(y,s)|\Big|\Big|\partial_{y}\psi(x,y,t,s)\Big|\, \dif w\\
&\leq\epsilon\iint_{D_T}\iint_{D_T}|\partial_{y}u_\epsilon(y,s)|\left|\partial_{y}\psi(x,y,t,s)\right|\,
\dif w\\
&\leq CT\,|u_0|_{BV(\Lambda)}\frac{\epsilon}{\rho},
\end{split}
\end{equation*}
where the last inequality is a consequence of the estimate
$|u_\eps(\cdot,t)|_{BV(\Lambda)}\leq|u_0|_{BV(\Lambda)}$ and \eqref{test_func}.

To estimate the left hand side (L.H.S.) of \eqref{dbl}, note that
\begin{equation*}
\begin{split}
&-|u_\epsilon(y,s)-u(x,t)|\phi'(t)\\
&\geq -|u_\epsilon(x,t)-u(x,t)|\phi'(t) -|u_\epsilon(x,s)-u_\epsilon(x,t)||\phi'(t)|-|u_\epsilon(y,s)-u_\epsilon(x,s)||\phi'(t)|,
\end{split}
\end{equation*}
and that
\begin{equation*}
\begin{split}
\iint_{D_T}\iint_{D_T}|u_\epsilon(x,s)-u_\epsilon(x,t)|\,\hat\omega_\rho\left(\frac{x-y}{2}\right)\omega_{\delta}\left(\frac{t-s}{2}\right)\,|\phi'(t)|\ \dif
w\stackrel{\delta\rightarrow 0}{\longrightarrow}0
\end{split}
\end{equation*}
and
\begin{equation*}
\begin{split}
&\iint_{D_T}\iint_{D_T}|u_\epsilon(y,s)-u_\epsilon(x,s)|\,\hat\omega_\rho\left(\frac{x-y}{2}\right)\omega_{\delta}\left(\frac{t-s}{2}\right)\,|\phi'(t)|\
\dif w\leq C\,T |u_0|_{BV}\,\rho.
\end{split}
\end{equation*}
Hence we conclude after sending $\delta\ra0$ that
$$-\iint_{D_T}|u_\epsilon(x,t)-u(x,t)|\,\phi'(t)dxdt-C\rho \leq
L.H.S.\ (\leq R.H.S.).$$ The results then follows by setting
$\rho=\sqrt{\epsilon}$ and $\phi=\chi_\mu$ as in \eqref{chi_fun}, and conclude as in the proof
of Theorem \ref{th:uniqueness}: Sending $\mu\ra0$, setting $t_2=t$, and using part \emph{iii)} in Definition \ref{def:alibaud_periodic} to send
$t_1\rightarrow0$.

\section*{Acknowledgement}
We thank Prof.~Chi-Wang Shu for suggesting the topic of this paper to us.

\end{document}